\documentclass[a4paper,reqno]{amsart}

\usepackage{amsmath,amsfonts,amsthm,amssymb,latexsym,mathrsfs,enumerate,pgfplots,tikz-cd,tkz-graph,multirow,booktabs,mathtools}
\usepackage[colorlinks=true,linkcolor=blue,citecolor=red,urlcolor=violet,bookmarksdepth=subsection,final]{hyperref}

\usepackage{mnotes}

\usetikzlibrary{decorations.markings,automata,arrows,positioning,calc}

\usepackage{graphicx}
\graphicspath{ {.} }

\DeclareMathOperator{\diag}{\mathrm{diag}}

\DeclareMathOperator{\coker}{\mathrm{coker}}
\DeclareMathOperator{\Hom}{\mathrm{Hom}}

\DeclareMathOperator{\gln}{\mathrm{GL}_n}

\DeclareMathOperator{\Aut}{\mathrm{Aut}}

\DeclareMathOperator{\rank}{\mathrm{rank}}
\DeclareMathOperator{\tor}{\mathrm{Tor}}

\DeclareMathOperator{\modulo}{\mathrm{mod}}
\DeclareMathOperator{\sgn}{\mathrm{sgn}}

\newcommand{\nn}{\mathbb{N}}
\newcommand{\zz}{\mathbb{Z}}
\newcommand{\qq}{\mathbb{Q}}

\newcommand{\cc}{\mathbb{C}}
\newcommand{\pp}{\mathbb{P}}
\newcommand{\ee}{\mathbb{E}}
\newcommand{\dd}{\mathbb{D}}
\newcommand{\hh}{\mathbb{H}}
\newcommand{\g}{\widehat{\mathbb{G}}}

\newcommand{\ppm}{\mathcal{P}_{\bar m}}
\newcommand{\oo}{\mathcal{O}}
\newcommand{\cs}{\mathrm{C}^*}

\makeatletter
\@namedef{subjclassname@2020}{%
  \textup{2020} Mathematics Subject Classification}
\makeatother

\newtheorem {theorem}{Theorem}[section]

\newtheorem {proposition}[theorem]{Proposition}

\newtheorem {conjecture}[theorem]{Conjecture}

\theoremstyle {definition}

\newtheorem {example}[theorem]{Example}
\newtheorem {definition}[theorem]{Definition}
\newtheorem {remark}[theorem]{Remark}
\newtheorem {question}[theorem]{Question}
\newtheorem {notation}[theorem]{Notation and terminology}

\numberwithin{equation}{section}

\title{Operator $K$-theoretic analysis of random adjacency matrices}

\author[B.~Jacelon]{Bhishan Jacelon}
\address[B.~Jacelon]{
Institute of Mathematics of the Czech Academy of Sciences\\ \v{Z}itn\'{a} 25\\115 67 Prague 1\\Czech Republic}
\email{jacelon@math.cas.cz}

\author[I.~Khavkine]{Igor Khavkine}
\address[I.~Khavkine]{
Institute of Mathematics of the Czech Academy of Sciences\\ \v{Z}itn\'{a} 25\\115 67 Prague 1\\Czech Republic}
\address{Charles University, Faculty of Mathematics and Physics\\ Sokolovsk\'a 83\\ 186 75 Prague 8\\Czech Republic}
\email{khavkine@math.cas.cz}


\begin{document}

\begin{abstract} 
We appeal to results from combinatorial random matrix theory to deduce that various random graph $\mathrm{C}^*$-algebras are asymptotically almost surely Kirchberg algebras with trivial $K_1$. This in particular implies that, with high probability, the stable isomorphism classes of such algebras are exhausted by variations of Cuntz algebras that we term `Cuntz polygons'. These probabilistically generic algebras can be assembled into a Fra\"{i}ss\'{e} class whose limit structure $\mathbb{G}$ is consequently relevant to any $K$-theoretic analysis of finite graph $\mathrm{C}^*$-algebras. We also use computer simulations to experimentally verify the behaviour predicted by theory and to estimate the asymptotic probabilities of obtaining stable isomorphism classes represented by actual Cuntz algebras. These probabilities depend on the frequencies with which the Sylow $p$-subgroups of $K_0$ are cyclic and in some cases can be computed from existing theory. For random symmetric $r$-regular multigraphs, current theory can describe these frequencies for finite sets of odd primes $p$ not dividing $r-1$. A novel aspect of the collected data is the observation of new heuristics outside of this case, leading to a conjecture for the asymptotic probability of these graphs yielding $\mathrm{C}^*$-algebras stably isomorphic to Cuntz algebras. For other models of random multigraphs including Bernoulli (di)graphs, the data also allow us to estimate and heuristically explain the (surprisingly high) asymptotic probabilities of \emph{exact} isomorphism to a Cuntz algebra. Recognising the role played by Cuntz--Krieger algebras in the theory of symbolic dynamics, we also collect supplemental data to estimate (and in some cases, actually compute) the asymptotic probability of a random subshift of finite type being flow equivalent to a full shift.
\end{abstract}

\maketitle

\section{Introduction} \label{section:intro}

In this paper, we investigate properties of the asymptotic distributions of certain families of random integer-valued matrices. The matrices in question encode adjacency in graphs (or rather, directed multigraphs) created by randomly adding edges to a large number of vertices (see Section~\ref{section:models}), and thus offer a means of generating random graph $\cs$-algebras (see Section~\ref{section:polygons}). The present work is therefore a continuation of \cite{Jacelon:2023aa}, which initiated the study of random constructions of several classes of $\cs$-algebras of interest to the Elliott classification programme.

The graphs that we consider are with high probability strongly connected. The associated $\cs$-algebras are therefore purely infinite and simple, so they fall under the remit of the Kirchberg--Phillips classification theorem \cite[Theorem 4.2.1]{Phillips:2000fj} and indeed R{\o}rdam's earlier classification \cite[Theorem 6.5]{Rordam:1995aa} of simple Cuntz--Krieger algebras (see Section~\ref{section:polygons}). In other words, their (stable) isomorphism classes are determined by operator $K$-theory. Here, \emph{stable isomorphism} means isomorphism up to tensoring with the $\cs$-algebra of compact operators on a separable, infinite-dimensional Hilbert space. Developed from algebraic topology, \emph{$K$-theory} $(K_0(\cdot),K_1(\cdot))$ is a powerful tool for analysis in the category of $\cs$-algebras. Its role as a classifying functor is what allows us to provide $\cs$-algebraic interpretations of the various random events examined in the sequel. With this in mind, we adopt the language of graph $\cs$-algebras throughout the article. 

In \cite[\S6]{Jacelon:2023aa}, it was explained how some of the powerful machinery developed by Wood \cite{Wood:2017tk} could be adapted to $K$-theoretic analysis of graph algebras. We considered multigraphs $E$ built from unions of perfect matchings, and observed that the distributions of Sylow subgroups of random cokernels established in \cite{Nguyen:2018vh} were also present for $K_0(\cs(E))$, which is given by the cokernel of the transpose of the multigraph's adjacency matrix shifted by the identity. The original motivation for the current work was to experimentally verify the $K_0$-behaviour predicted by this theory. We were indeed able to do this (see Section~\ref{section:stats}), but what immediately became apparent from the data was something overlooked in \cite{Jacelon:2023aa}, namely, asymptotically almost-sure triviality of the $K_1$-group, which is given by the kernel of the shifted, transposed adjacency matrix.

The singularity problem for random matrices has a storied history in combinatorics (see the discussion and references in \cite{Nguyen:2018vh}). It posits that sufficiently random square matrices are nonsingular with probability approaching $1$ as the matrix size grows to infinity. Our working hypothesis is that this should be applicable not just to the adjacency matrices $A_E$ of appropriate random graphs $E$, but equally well to $M_E=A_E^t-I$. As alluded to earlier, this matrix determines $K_0(\cs(E))=\coker M_E$ and $K_1(\cs(E))=\ker M_E$ (see the discussion surrounding \eqref{eqn:kt}), and its nonsingularity exactly means triviality of $K_1$. We will see in Section~\ref{section:models} that this applies to \emph{Bernoulli digraphs} $\dd_{n,q}$ in which each possible edge (allowing loops but not multiple edges) occurs independently with fixed probability $q$, and also to random $r$-regular multigraphs $\g_{n,r}$ built from perfect matchings.

Because of this $K_1$-triviality, and in light of $K$-theoretic classification, the stable isomorphism classes of our random graph algebras (and, given the scope of the singularity problem, likely more generally) can with high probability be described by rather simple directed multigraphs. These representatives, which we call \emph{Cuntz polygons} (see Figure~\ref{fig:turtle}), are generalisations of Cuntz algebras in the sense that $K_0$ need not be cyclic but can be any finite abelian group. Stable isomorphism to an actual Cuntz algebra is detected by cyclicity of $K_0$. As it turns out, for large Bernoulli digraphs this happens about $85\%$ of the time, a probability which can be computed asymptotically from existing theory (see Theorem~\ref{thm:bernoulli}) and which is supported by the data (Table~\ref{table:dq}). For random $r$-regular multigraphs, there are currently two obstacles to computing the asymptotic probability of the graph algebra being stably isomorphic to a Cuntz algebra. They both stem from the fact that cyclicity of $K_0$ is equivalent to cyclicity of the Sylow $p$-subgroups of $K_0$ for all primes $p$. First, although the probability of cyclicity of Sylow $p$-subgroups is known for odd primes $p$ that do not divide $r-1$, it is unknown for $p\mid 2(r-1)$ (that is, for $p=2$ or $p$ an odd prime dividing $r-1$). Second, the known theory covers the case of only \emph{finitely many} primes, not all primes simultaneously.  See the discussion around Theorem~\ref{thm:regular}, which is adapted from \cite[Theorem 1.6]{Nguyen:2018vh}.

However, we conjecture from the data that the asymptotic probability of the graph algebra being stably isomorphic to a Cuntz algebra still converges to the product over all primes $p$ of the asymptotic probability of cyclicity of the Sylow $p$-subgroup of $K_0$ (see Table~\ref{table:adams}). When the regularity degree is one more than a power of two, we estimate this product to be about $40\%$ (see Table~\ref{table:2k1}). As for cyclicity of the Sylow $p$-subgroup of $K_0(\g_{n,r})$ for large $n$ when $p\mid 2(r-1)$, the data reveal new heuristics (see Table~\ref{table:adams},  Table~\ref{table:newdist} and Conjecture~\ref{conjecture:regular}).

On the other hand, the \emph{symmetric} versions $\ee_{n,q}$ of $\dd_{n,q}$ (that is, \emph{Erd\H{o}s--R\'{e}nyi graphs} allowing loops) lead to $\cs$-algebras $\cs(\ee_{n,q})$ for which \emph{all} Sylow $p$-subgroups of $K_0(\cs(\ee_{n,q}))$, including $p=2$, follow the limiting behaviour \eqref{eqn:ncyclic}, \eqref{eqn:cyclicn} and \eqref{eqn:allcyclic}. This is proved in Theorem 3.12 of Wood's article \cite{Wood:2023aa}, where it is also stated that forthcoming work will extend this result to infinitely many primes. Moreover, asymptotic triviality of $K_1$ also holds for these algebras, for the same reason as for the algebras $\cs(\g_{n,r})$ (namely, \cite[Corollary 4.2]{Nguyen:2018vh}). One would therefore expect about $79\%$ of the algebras $\cs(\ee_{n,q})$ for large $n$ to be stably isomorphic to Cuntz algebras.

Cuntz--Krieger algebras are intimately connected to the theory of symbolic dynamics. If the (directed multi-) graph $E$ is strongly connected and does not consist of a single cycle (equivalently, if the adjacency matrix is irreducible and is not a permutation matrix), then the stable isomorphism class of the graph algebra $\cs(E)$ (which is a Cuntz--Krieger algebra) is an invariant for the flow equivalence class of the associated edge shift $\sigma \colon \Sigma_E \to \Sigma_E$. Here, $\Sigma_E$ is the two-sided infinite path space
\begin{equation} \label{eqn:shift}
\Sigma_E = \left\{(e_i)_{i\in\zz}\in (E^1)^\zz \mid s(e_{i+1})=r(e_i) \text{ for every } i\in\zz\right\}
\end{equation}
(see Section~\ref{section:polygons} for notation) and $\sigma$ is the shift map that sends $(e_i)_{i\in\zz}$ to $(e_{i+1})_{i\in\zz}$. \emph{Flow equivalence} between topological dynamical systems $(\Sigma_E,\sigma_E)$ and $(\Sigma_F,\sigma_F)$ is a weakening of conjugacy that refers to equivalence of their suspension flows (see, for example, \cite[\S13.6]{Lind:1995wp}). In our setting, flow equivalence can be characterised by certain graph \emph{moves} (see \cite{Sorensen:2013aa,Carlsen:2019aa}) and implies the existence of a \emph{diagonal-preserving} stable isomorphism between the graph algebras (see \cite[\S4]{Cuntz:1980hl} and \cite{Carlsen:2019aa}). Mere stable isomorphism of the associated graph algebras (or in other words, isomorphism of their $K_0$ groups) is strictly weaker than flow equivalence. By \cite{Franks:1984aa}, the necessary and sufficient additional data is the sign of the determinant of $I-A_E$, which when combined with the $K_0$ group yields the \emph{signed Bowen--Franks group}
\begin{equation} \label{eqn:bf}
\mathrm{BF}_+(E)=\sgn(\det(I-A_E))\coker(I-A_E),
\end{equation}
where we notationally attach the sign label `$+$' to the group $\coker(I-A_E)$ if $\det(I-A_E)>0$, or `$-$' if $\det(I-A_E)<0$. For the graph $E$ with one vertex and $n$ loops (whose shift space $\Sigma_E$ is the full shift on an alphabet of $n$ symbols and whose graph algebra is the Cuntz algebra $\oo_n$), $\mathrm{BF}_+(E)=-\zz/(n-1)\zz$. Collecting this sign data in conjunction with a tally of cyclicity therefore allows us to estimate the asymptotic probabilities of a random subshift of finite type (that is, the edge shift associated to a random graph) being flow equivalent to a full shift. For Bernoulli digraphs $\dd_{n,q}$, the percentage probability appears to be about $42$ (see Conjecture~\ref{conjecture:fullshift} and Table~\ref{table:dnqextra}). In certain special cases (Theorem~\ref{thm:bf}), we can prove that this is indeed the right number (thus providing a question to an answer of D.~Adams \cite{Adams:1995aa}). The asymptotic behaviour of the determinant for $\ee_{n,q}$ and $\g_{n,r}$, about which we have formulated and left open Question~\ref{q1}, is less clear. However, the full-shift probability for large graphs is still nontrivial (see Tables~\ref{table:enqextra},~\ref{table:gnrextra}).

As for the property of being \emph{exactly} (rather than stably) isomorphic to a Cuntz algebra, the asymptotic probability of this event does not appear to be readily deducible from existing random matrix theory. (There is also no available dynamical interpretation of the corresponding equivalence relation for the associated edge shifts, although there is a geometric description via graph moves; see \cite{Arklint:2022aa}.) Nonetheless, by keeping track of the class of the unit we are able to provide estimates (together with heuristic explanations; see Conjecture~\ref{conjecture:bernoulli}, Table~\ref{table:dnqextra} and Table~\ref{table:enqextra}). The surprising conclusions are that, if $n$ is large, then $\cs(\dd_{n,q})$ is about $51\%$ likely to be isomorphic to a Cuntz polygon and is about $44\%$ likely to be isomorphic to an actual Cuntz algebra. For $\cs(\ee_{n,q})$, these numbers are $61\%$ and $48\%$, respectively. In short, it's about a coin toss whether a chance encounter with a large graph $\cs$-algebra is in fact a meeting with a Cuntz algebra.

How should this (and our other probabilistic observations) be interpreted? The first thing we can say is that theorems about Cuntz polygons become asymptotically almost-sure theorems about the stable isomorphism classes of random graph $\cs$-algebras. For example, in Section~\ref{subsection:fraisse} we make the observation that Cuntz polygons $\ppm$ form a \emph{Fra\"{i}ss\'{e} class}. The Fra\"{i}ss\'{e} limit $\mathbb{G}$ of this class (which by \cite{Szymanski:2002yq} is the graph algebra of an infinite, strongly connected graph) is then a structure that is asymptotically almost surely universal and homogeneous for sufficiently random unital graph algebras. See Theorem~\ref{thm:homogeneous} and also Section~\ref{subsection:contiguity}, in which we emphasise the point that, while we have opted to work with distributions amenable to calculations of probability, similar conclusions should hold for any distribution that shares the same asymptotically almost-sure events (so-called \emph{contiguous} distributions). When applied to Bernoulli digraphs as in Theorem~\ref{thm:uniform}, this general slogan materialises into the concrete observation that \emph{most} of the $\cs$-algebras associated to finite digraphs on large vertex sets with suitably many edges are stably isomorphic to Cuntz polygons. The other message is that complex structures are rare, or put another way, the objects that we tend to encounter are the ones with minimal symmetry. On the flip side, when there is imposed extra structure (such as graph regularity), it seems that we tend to see more complicated automorphism groups with greater frequency. (It is also interesting that $\cs(\g_{n,r})$ being isomorphic to a Cuntz polygon is a very rare event (see Table~\ref{table:gnrextra}). Question~\ref{q2}, whether this probability asymptotes to zero, is left open.) The central point of the present work is that for graph $\cs$-algebras these various statements can be made quantifiably precise.

We have gathered the data reported in Section~\ref{section:stats} by running computer simulations. Our code generates samples of random graphs $\g_{n,r}$, $\dd_{n,q}$ or $\ee_{n,q}$ (our typical sample size being $m=10^5$) and collects $K$-theoretic data. 
The primary tool for computing cokernels of integer matrices, and hence the $K$-theory of graph algebras, is the Smith normal form (SNF) algorithm. It has been used extensively in the literature (see, for example, \cite{Rordam:1995aa, Cornelissen:2008aa, Eilers:2021aa}) and is an integral piece of our analysis of the graphs. The SNF algorithm provides, for a given $M_E\in M_n(\zz)$, a diagonalisation $UM_EV=D=\diag(d_1,\dots,d_n)$, where $U,V\in\gln(\zz)$ and the $d_i$ are integers with $d_i|d_{i+1}$ for $1\le i\le n-1$. Then,
\begin{equation} \label{eqn:snf}
\coker M_E \cong \coker D \cong \zz/d_1\zz \oplus\dots \oplus \zz/d_n\zz.
\end{equation}
From this, we can also determine the structure of the $p$-Sylow subgroup of $\coker M_E \cong K_0(\cs(E))$ for any prime $p$. We also record:
\begin{itemize}
\item whether or not each graph is strongly connected;
\item whether $K_0$ is cyclic, by removing any $0$ entries from the list $L=[d_1,\dots,d_n]$ (storing these as $K_1$), and also any $1$s (as these do not contribute to the cokernel), then checking whether there remains at most one entry in $L$;
\item the rank of $K_1$ (in other words, the number of $0$ entries removed from $L$);
\item tallies of cyclicity of Sylow $p$-subgroups for specified primes $p$, as well as instances where these subgroups are of the form $\zz/(p^N\zz)$ or $(\zz/p\zz)^N$, for integers $N$ that are small enough for these events to be statistically observable (see \eqref{eqn:ncyclic} and \eqref{eqn:cyclicn});
\item whether or not the class of the unit in $K_0$ is in the same automorphism orbit as the `canonical' class, which in instances of trivial $K_1$ means that we have exact rather than stable isomorphism between $\cs(E)$ and the corresponding Cuntz polygon (see Section~\ref{subsection:unit});
\item the sign of the determinant of $I-A_E=-M^t_E$ (and hence, if $K_0$ is also cyclic, whether or not the corresponding edge shift is flow equivalent to a full shift);
\item $99\%$ confidence intervals for the various proportions above, using the normal approximation to the binomial distribution.
\end{itemize}
The simulation and analysis code is written in Python. For efficiency reasons, the SNF is computed by calling the PARI library~\cite{PARI2}, which is written in C. Both the code and experimental data are available on GitHub.\footnote{\url{https://github.com/bjacelon/random-graph-k-theory}} There are three reasons for our decision to include a substantial offering of empirical data. First, we are demonstrating the reliability of our experimental process by comparing with behaviour predicted by theory, whenever the appropriate theory is available. Next, in situations where there is as yet no theoretical prediction, the estimates that we extract from the data can then be expected to carry some weight. Last, but not least, we wish to document the present mathematical endeavour as the conversation between theory, practice and scientific discovery that it truly was.

We have organised the article as follows. Section~\ref{section:polygons} provides some background on the $\cs$-algebras associated to finite graphs, as well as their $K$-theoretic classification, and introduces Cuntz polygons. Section~\ref{section:models} describes the relevant random graph models and some asymptotic behaviour guaranteed by random matrix theory. Finally, Section~\ref{section:stats} contains empirical data that either experimentally verify the theory or provide estimates for probabilities that are not yet theoretically computable.

\subsection*{Acknowledgements} BJ was supported by the GA\v{C}R project 22-07833K and partially supported by the Simons Foundation Award No 663281 granted to the Institute of Mathematics of the Polish Academy of Sciences for the years 2021--2023. He is grateful to the organisers of the stimulating `Graph Algebras' and `Generic Structures' conferences held at the Bedlewo Centre in 2023. IK is partially supported by the Praemium Academiae of M.~Markl and by Czech science foundation (GA\v{C}R) under the grant GA22-00091S. This collaboration would not have been possible without the excellent working environment at the Institute of Mathematics of the Czech Academy of Sciences (RVO: 67985840). The computer code was written with some facilitation by ChatGPT~\cite{ChatGPT}. Both authors are indebted to S{\o}ren Eilers and the anonymous referee of an earlier version of this article for pointing out the connections to symbolic dynamics and for encouraging a deeper dive into the question of exact vs stable isomorphism between random graph algebras and Cuntz polygons.

\section{Graph \texorpdfstring{$\mathrm{C}^*$}{C*}-algebras} \label{section:polygons}

\subsection{Graph algebras}

The `graph' in `graph $\cs$-algebra' might in other contexts be called a directed (multi)graph. For clarity, we begin by specifying the kinds of graphs that are amenable to our analysis and then we explain how $\cs$-algebras are attached to them.

A \emph{directed graph} $E$ consists of a vertex set $E^0$, an edge set $E^1$, and range and source maps $r,s\colon E^1\to E^0$. A \emph{path} in $E$ is a (possibly finite) sequence of edges $(\alpha_i)_{i\ge1}$ such that $r(\alpha_i)=s(\alpha_{i+1})$ for every $i$, and a \emph{cycle} is a finite path $(\alpha_i)_{i=1}^n$ whose initial and final vertices coincide and with no other repeated vertices. A vertex $v\in E^0$ is called a \emph{sink} if $s^{-1}(v)=\emptyset$. Every graph considered in this article will be finite and (except for some of the graphs mentioned in Remark~\ref{remark:logn} and Table~\ref{table:dlog}) will with high probability have no sinks.

We recall the definition of the graph algebra associated to $E$. (Note that a different convention is used in \cite{Raeburn:2005aa}. What follows is the definition that appears, for example, in \cite{Bates:2000fk}.) A \emph{Cuntz--Krieger $E$-family} associated to a directed graph $E=(E^0,E^1,s,r)$ is a set
\begin{equation} \label{eqn:gen}
\{p_v \mid v\in E^0\} \cup \{s_e \mid e\in E^1\},
\end{equation}
where the $p_v$ are mutually orthogonal projections and the $s_e$ are partial isometries satisfying:
\begin{align} \label{eqn:rel}
s_e^*s_f &= 0 &&\forall\: e,f\in E^1 \text{ with } e\ne f \nonumber\\
s_e^*s_e &= p_{r(e)}
&&\forall\: e\in E^1 \nonumber\\
s_es_e^* &\le p_{s(e)} &&\forall\: e\in E^1 \nonumber\\
p_v &= \sum_{e\in s^{-1}(v)} s_es_e^* &&\forall\: v\in E^0 \:\text{that is not a sink}.
\end{align}

The \emph{graph algebra} $\cs(E)$ is the universal $\cs$-algebra with generators \eqref{eqn:gen} satisfying the relations \eqref{eqn:rel}. Every graph algebra $\cs(E)$ is separable, nuclear and satisfies the universal coefficient theorem (UCT) (see, for example, \cite[Remark 4.3]{Raeburn:2005aa} and \cite[Lemma 3.1]{Raeburn:2004tg}), and $\cs(E)$ is unital if and only if $E^0$ is finite, in which case $1_{\cs(E)}=\sum_{v\in  E^0}p_v$.

\subsection{Cuntz--Krieger algebras}

By a \emph{Cuntz--Krieger algebra} we mean the graph $\cs$-algebra of a finite graph without sinks. This is not the original definition that appears in \cite{Cuntz:1980hl}, but is equivalent to it (even in the absence of Cuntz and Krieger's condition (I); see \cite[Definition 2.3]{Arklint:2015tu} and the ensuing discussion, and also \cite[Theorem 3.12]{Arklint:2015tu}).

If $E$ is such a graph, and $A_E$ is its adjacency matrix
\[
A_E(v,w) = |\{e\in E^1 \mid s(e)=v, r(e)=w\}|,
\]
then the $K$-theory of $\cs(E)$ can be computed as follows. From the proof of \cite[Theorem 3.2]{Raeburn:2004tg}, the map $\zz^{E^0}\to K_0(\cs(E))$ that sends the standard basis vector $e_v$ to $[p_v]$ is surjective with kernel isomorphic to $(A_E^t-I)\zz^{E^0}$. In particular,
\begin{equation} \label{eqn:kt}
(K_0(\cs(E)),[1_{\cs(E)}]_0,K_1(\cs(E))) \cong (\coker(A_E^t-I),[(1,\dots,1)], \ker(A_E^t-I)).
\end{equation}

A \emph{UCT Kirchberg algebra} is by definition a purely infinite, simple, nuclear, separable $\cs$-algebra that satisfies the universal coefficient theorem. By the Kirchberg--Phillips theorem (see \cite[Theorem 4.2.1]{Phillips:2000fj}, and also \cite{Kirchberg:2000kq} and \cite[Chapter 8]{Rordam:2002yu}), these algebras are classified up to stable isomorphism by $K$-theory. Further, specifying the class $[1]_0$ of the unit determines the isomorphism class of the algebra (see \cite[Theorem 4.2.4]{Phillips:2000fj}).

The following is a consequence of \cite[Propositions 5.1 and 5.3]{Bates:2000fk}.

\begin{proposition} \label{prop:ckk}
If $E$ is strongly connected and does not consist of a single cycle (so in particular, every cycle has an exit), then the Cuntz--Krieger algebra $\cs(E)$ is purely infinite and simple, so is a UCT Kirchberg algebra.
\end{proposition}

\begin{remark}
\begin{enumerate}[1.]
\item We do not actually need the full strength of the Kirchberg--Phillips theorem in this article, and could instead appeal to R{\o}rdam's earlier classification \cite[Theorem 6.5]{Rordam:1995aa} of simple Cuntz--Krieger algebras.
\item While the graph algebras of primary interest to us will be covered by Proposition~\ref{prop:ckk}, it should be noted that the full class of unital graph algebras is classifiable by an invariant called ordered reduced filtered $K$-theory that accounts for the ideal structure of the algebras (see \cite{Eilers:2021aa}).
\end{enumerate}
\end{remark}

\subsection{Cuntz polygons} \label{subsection:polygons}

\begin{definition} \label{def:polygon}
For $n\in\nn$ and $\bar m = (m_1,\dots,m_n)\in\nn^n$, let $E_{\bar m}$ be the graph with vertices
\[
E_{\bar m}^0=\{v_1,\dots,v_n\}
\]
and edges
\[
E_{\bar m}^1 = \{e_{i,j} \mid 1\le i \le n, 1\le j \le m_i\}\cup\{l_1,\dots,l_n\}
\]
where, for $1\le i\le n$ and $1\le j\le m_i$,
\[
r(l_i)=s(l_i)=v_i,\: r(e_{i,j})=v_{i},\: s(e_{i,j})=v_{i-1}\: (\modulo n).
\]
We call the graph algebra $\ppm:=\cs(E_{\bar m})$ a \emph{Cuntz $n$-gon}. A \emph{Cuntz polygon} is a Cuntz $n$-gon for some $n$.
\end{definition} 

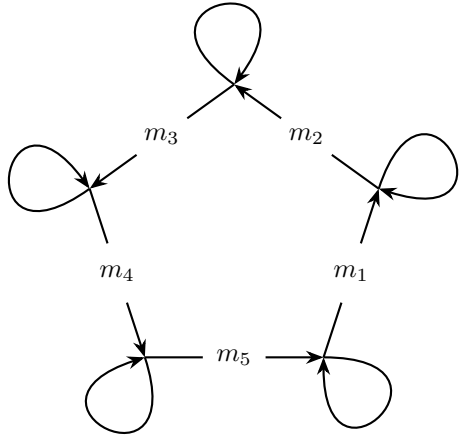
\begin{figure}[!htbp]
\centering
\begin{tikzpicture} 
\begin{scope}[every node/.style={coordinate,draw}]
    \node (A) at (canvas polar cs: angle=18, radius=2cm) {$v_1$};
    \node (B) at (canvas polar cs: angle=90, radius=2cm) {$v_2$};
    \node (C) at (canvas polar cs: angle=162, radius=2cm) {$v_3$};
    \node (D) at (canvas polar cs: angle=234, radius=2cm) {$v_4$};
    \node (E) at (canvas polar cs: angle=306, radius=2cm) {$v_5$};
\end{scope}

\begin{scope}[>={Stealth[black]},
              every node/.style={fill=white,circle},
              every edge/.style={draw=black, thick}]
    \path [->] (A) edge node {$m_2$} (B);
    \path [->] (B) edge node {$m_3$} (C);
    \path [->] (C) edge node {$m_4$} (D);
    \path [->] (D) edge node {$m_5$} (E);
    \path [->] (E) edge node {$m_1$} (A);
    \path [->, every loop/.style={min distance=2cm, in=-18,out=72}] (A) edge[loop] (A);
    \path [->, every loop/.style={min distance=2cm, in=54,out=144}] (B) edge[loop left] (B); 
    \path [->, every loop/.style={min distance=2cm, in=126,out=216}] (C) edge[loop right] (C); 
    \path [->, every loop/.style={min distance=2cm, in=198,out=288}] (D) edge[loop below] (D); 
    \path [->, every loop/.style={min distance=2cm, in=270,out=360}] (E) edge[loop right] (E); 
\end{scope}
\end{tikzpicture}
\caption{The graph $E_{(m_1,m_2,m_3,m_4,m_5)}$ of the Cuntz pentagon $\mathcal{P}_{(m_1,m_2,m_3,m_4,m_5)}$.} \label{fig:turtle}
\end{figure}

The reason for the terminology, apart from the shapes of their associated graphs (see Figure~\ref{fig:turtle}), is that Cuntz polygons are generalisations of Cuntz algebras. Note indeed that a Cuntz $1$-gon $\mathcal{P}_{(m)}$ is isomorphic to $\oo_{m+1}$. Moreover, we have the following.

\begin{proposition} ~\label{prop:kpoly}
\begin{enumerate}[1.]
\item \label{poly1} If $A$ is a Cuntz polygon $\cs(E_{\bar m})$, then $A$ is a UCT Kirchberg algebra with
\begin{equation} \label{eqn:cpk}
(K_0(A),[1_A]_0,K_1(A)) \cong \left(\bigoplus_{i=1}^n\zz/m_i\zz,(1,\dots,1),0\right).
\end{equation}
\item \label{poly2} A Cuntz polygon $\ppm$ is stably isomorphic to a Cuntz algebra if and only if $K_0(\ppm)$ is cyclic, which happens if and only if the components $m_1,\dots,m_n$ of $\bar m$ are pairwise coprime.
\item \label{poly3} If $E$ is a strongly connected finite graph whose adjacency matrix $A_E$ is not a permutation matrix and is such that $A_E^t-I$ is nonsingular, then $\cs(E)$ is stably isomorphic to $\ppm$ for some $\bar m$.
\end{enumerate}
\end{proposition}

\begin{proof}
\begin{enumerate}[1.]
\item Since $E_{\bar m}$ is strongly connected and does not consist of a single cycle, $A$ is a UCT Kirchberg algebra by Proposition~\ref{prop:ckk}. The adjacency matrix of $E_{\bar m}$ is $A_{E_{\bar m}}=I+PD$, where $D$ is the diagonal matrix $\diag(m_1,\dots,m_n)$ and $P$ is the permutation matrix whose action on the standard basis of $\zz^{E^0}$ is given by $Pe_{v_i}=e_{v_{i-1}}$ (mod $n$). So, $A^t_{E_{\bar m}}-I=DP^t$, which implies that $\ker(A^t_{E_{\bar m}}-I)=0$ and $\coker(A^t_{E_{\bar m}}-I)=\coker D \cong \bigoplus_{i=1}^n\zz/m_i\zz$. Therefore, \eqref{eqn:cpk} follows from \eqref{eqn:kt}.

\item The second assertion follows from the Kirchberg--Phillips (or R{\o}rdam) classification theorem together with the fact that the Cuntz algebra $\oo_n$ is a UCT Kirchberg algebra with $K$-theory
\[
(K_0(\oo_n),[1_{\oo_n}]_0,K_1(\oo_n)] \cong (\zz/(n-1)\zz,1,0)
\]
(see \cite{Cuntz:1977qy} and \cite{Cuntz:1981kq}). Note also that $\bigoplus_{i=1}^n\zz/m_i\zz$ is cyclic if and only if the $m_i$ are pairwise coprime, and in this case, $\ppm$ is in fact isomorphic to $\oo_{1+\prod_{i=1}^nm_i}$.

\item Finally, if $A_E$ is not a permutation matrix, then $E$ does not consist of a single cycle. By Proposition~\ref{prop:ckk}, $\cs(E)$ is therefore a UCT Kirchberg algebra. By \eqref{eqn:kt} and \eqref{eqn:cpk}, there exists some Cuntz polygon that has the same $K$-theory as $\cs(E)$, which is then in the same stable isomorphism class by the Kirchberg--Phillips (or R{\o}rdam) classification theorem. \qedhere
\end{enumerate}
\end{proof}

\subsection{Exact vs stable isomorphism} \label{subsection:unit}

While Cuntz polygons cover all \emph{stable} isomorphism classes of purely infinite simple Cuntz--Krieger algebras with trivial $K_1$, they are far from being able to capture all \emph{isomorphism} classes. The reason of course is the position of the unit in $K_0$, information which is encoded by the potentially complicated structure of the graph. In fact, the class of the unit can be arbitrary: by \cite[Proposition 3.9]{Eilers:2016wz}, for any finitely generated abelian group $G$, distinguished element $g\in G$ and free abelian group $F$ with $\rank F = \rank G$ (for example, $F=0$ and $G$ any finite abelian group), there exists a finite graph $E$ such that $\cs(E)$ is a Kirchberg algebra with
\[
\left(K_0(\cs(E)), [1_{\cs(E)}]_0, K_1(\cs(E))\right) \cong (G,g,F).
\]

But actually, we can keep track of the class of the unit via the matrix $U$ that appears in the SNF decomposition $UM_EV=D$ of $M_E=A_E^t-I$. More precisely, and elaborating on \eqref{eqn:snf}, the map $\zz^n\to\zz^n/D\zz^n=\coker(D)$, $x\mapsto [Ux]$ is surjective with kernel $M_E\zz^n$, so gives an isomorphism between $\coker(M_E)$ and $\coker(D)\cong \zz/d_1\zz \oplus\dots \oplus \zz/d_n\zz$. In light of \eqref{eqn:kt}, the class of the unit under this isomorphism is $\left[U\left(\begin{smallmatrix}1\\ \vdots\\ \\1\end{smallmatrix}\right)\right]$.

Suppose now that $E$ satisfies the hypotheses of Proposition~\ref{prop:kpoly}.\ref{poly3}, so that $\cs(E)$ is stably isomorphic to the corresponding Cuntz polygon $\mathcal{P}_{\bar d}:=\mathcal{P}_{(d_1,\dots,d_n)}$. By Proposition~\ref{prop:kpoly}.\ref{poly1}, the class of the unit for $\mathcal{P}_{\bar d}$ is  $\left[\left(\begin{smallmatrix}1\\ \vdots\\ \\1\end{smallmatrix}\right)\right]$. Comparing the automorphism orbits of these two vectors (by incorporating the description presented in \cite{Dutta:2011aa,Schwachhofer:1999aa} into our computer code), we can therefore decide whether or not there exists an isomorphism between $(K_0(\cs(E)),[1_{\cs(E)}])$ and $(K_0(\mathcal{P}_{\bar d}),[1_{\mathcal{P}_{\bar d}}])$, giving an isomorphism between $\cs(E)$ and $\mathcal{P}_{\bar d}$. In Section~\ref{subsection:exact}, we present a heuristic analysis of the asymptotic probability of this event for Bernoulli digraphs $\dd_{n,q}$ and their symmetric cousins $\ee_{n,q}$. Section~\ref{section:stats} contains experimental confirmation of the formulas displayed in Conjecture~\ref{conjecture:bernoulli}.

\subsection{The Fra\"{i}ss\'{e} theory of Cuntz polygons} \label{subsection:fraisse}

The category whose objects are finite abelian groups and whose morphisms are injective group homomorphisms forms a \emph{Fra\"{i}ss\'{e} class} (see, for example, the discussion and references in \cite[\S2]{Kostana:2023aa}). The content of this assertion, given the more immediately obvious properties of separability (that is, countably many isomorphism classes of objects), heredity (closure under substructures) and joint embedding (any two objects can both be embedded into a third), is the ability to amalgamate, that is, to complete commutative diagrams of the form
\[
\begin{tikzcd}[column sep=small, row sep=small]
& H_1 \arrow[hook,dashed,dr] & \\
G_1 \arrow[hook,ur] \arrow[hookrightarrow,dr] & & G_2 \\
& H_2 \arrow[hook,dashed,ur] &
\end{tikzcd}
\]
Attached to any Fra\"{i}ss\'{e} class $\mathcal{L}$ is its \emph{Fra\"{i}ss\'{e} limit} $\mathbb{F}$, the unique structure whose \emph{age} (the collection of its finite substructures) is $\mathcal{L}$, and which is \emph{homogeneous}, meaning that any two embeddings of an object of $\mathcal{L}$ into $\mathbb{F}$ must lie in the same automorphism orbit (where $\Aut(\mathbb{F})$ acts on embeddings by post-composition). For the class of finite abelian groups, this limit is the group $\mathbb{A}=\bigoplus_{i=1}^\infty\qq/\zz$ (see \cite[Proposition 6]{Kostana:2023aa}).

\begin{definition} \label{def:fraisse}
Let $\mathbb{G}$ be the stable UCT Kirchberg algebra with $(K_0(\mathbb{G}),K_1(\mathbb{G}))=(\mathbb{A},0)$.
\end{definition}

Now let us consider the category $\mathcal{P}$ whose objects are stabilised Cuntz polygons (that is, $\cs$-algebras of the form $\ppm\otimes\mathbb{K}$ for some $\bar m$, where $\mathbb{K}$ is the $\cs$-algebra of compact operators on a separable, infinite-dimensional Hilbert space) and whose morphisms are asymptotic unitary equivalence classes of \emph{$K$-embeddings}, by which we mean nonzero $^*$-homomorphisms that are embeddings at the level of $K$-theory. Note that, since the objects in $\mathcal{P}$ are simple, $K$-embeddings are embeddings. By Kirchberg's classification theorem (see \cite[Theorem 8.3.3]{Rordam:2002yu}), this is also a Fra\"{i}ss\'{e} class in the suitable $\cs$-algebraic sense (see \cite{Eagle:2016ww,Masumoto:2017wx,Jacelon:2021uc}). The $\cs$-algebra $\mathbb{G}$ defined above is the $\cs$-algebraic Fra\"{i}ss\'{e} limit of this class. $\mathbb{G}$ is:
\begin{itemize}
\item (by \cite{Szymanski:2002yq}) the graph $\cs$-algebra of an infinite (but row-finite) strongly connected graph;
\item \emph{universal}, by which we mean that it can be built as an inductive limit of objects in $\mathcal{P}$ in such a way that every object in $\mathcal{P}$ can be embedded into some finite stage of the inductive sequence (a fact which follows from Kirchberg's classification theorem since the corresponding statement is true of $\mathbb{A}$);
\item \emph{homogeneous}, meaning that for every object $B$ in $\mathcal{P}$, $\Aut(\mathbb{G})$ acts transitively on the space of embeddings of $B$ into $\mathbb{G}$.
\end{itemize}

In summary, we have the following.

\begin{theorem} \label{thm:homogeneous}
Let $B$ be a UCT Kirchberg algebra with $K_1(B)=0$ and $K_0(B)$ finite (in other words, $B$ is stably isomorphic to a Cuntz polygon). Then, there is an embedding $\varphi\colon B\to\mathbb{G}$ such that $K_0(\varphi)$ is injective. Moreover, if $\psi\colon B\to \mathbb{G}$ is another such embedding then there exists an automorphism $\alpha$ of $\mathbb{G}$ such that $\psi$ is asymptotically unitarily equivalent to $\alpha\circ\varphi$.
\end{theorem}

\begin{remark} \label{remark:amenable}
In contrast to the Fra\"{i}ss\'{e} limits described in \cite{Eagle:2016ww,Masumoto:2017wx,Jacelon:2021uc}, the $\cs$-algebra $\mathbb{G}$ is purely infinite rather than stably finite. In the stably finite setting, it is an open problem to decide whether the automorphism group of the $\cs$-algebraic limit structure is \emph{extremely amenable} (under the topology of pointwise convergence), meaning that every continuous action of the group on a nonempty compact space should admit a fixed point. This property is indeed a feature of $\Aut(A)$ when $A$ is a uniformly hyperfinite (UHF) $\cs$-algebra (see \cite[\S5]{Eagle:2016ww}), but it is unclear whether it holds when $A$ is the Jiang--Su algebra $\mathcal{Z}$ or any of the stably projectionless algebras considered in \cite{Jacelon:2021uc}. For $A=\mathbb{G}$, the following argument shows that $\Aut(\mathbb{G})$ is \emph{not} extremely amenable. By \cite{Deuber:1978aa} (a reference for which we thank Vadim Alekseev), the class of finite abelian groups does not have the Ramsey property, so by the KPT correspondence \cite{Kechris:2005ly}, $\Aut(K_0(\mathbb{G}))\cong\Aut(\mathbb{A})$ is not extremely amenable when equipped with the pointwise-convergence topology. By \cite[Corollary 8.4.10]{Rordam:2002yu}, there is an isomorphism $\Aut(\mathbb{G})/\Aut_0(\mathbb{G}) \to \Aut(K_0(\mathbb{G}))$ (where $\Aut_0(\mathbb{G})$ denotes the closed subgroup of asymptotically inner automorphisms of $\mathbb{G}$) which can moreover be checked to be continuous and open (so is an isomorphism of \emph{topological} groups). Since extreme amenability passes to quotients, we conclude that $\Aut(\mathbb{G})$ is not extremely amenable.
\end{remark}

\begin{remark} \label{remark:orbit}
Let $\mathcal{P}_\mathbb{G}$ denote the set of asymptotic unitary equivalence classes of the $\cs$-subalgebras of $\mathbb{G}$ that are stably isomorphic to $\ppm$ for some $\bar m$. More precisely, we consider subalgebras $C,D\subseteq \mathbb{G}$ to be equivalent if there exists an isomorphism $\alpha\colon C \to D$ (which, by homogeneity, can be extended to $\tilde{\alpha}\in\Aut(\mathbb{G})$) such that $\tilde{\alpha}|_C$ is asymptotically unitarily equivalent to the inclusion $C\hookrightarrow\mathbb{G}$ . Note that every $\cs$-algebra representing an equivalence class in $\mathcal{P}_\mathbb{G}$ is a UCT Kirchberg algebra, so is either stable or unital. Let us correspondingly separate $\mathcal{P}_\mathbb{G} = \mathcal{P}^0_\mathbb{G} \sqcup \mathcal{P}^1_\mathbb{G}$ into the classes of stable algebras and the classes of unital ones. Fix a set $\mathcal{C}_\mathbb{G} \subseteq \mathcal{P}^1_\mathbb{G}$ of representatives of isomorphism classes, one for each Cuntz polygon, so that their stabilisations $\mathcal{C}_\mathbb{G} \otimes \mathbb{K} \subseteq \mathbb{G} \otimes \mathbb{K} \cong \mathbb{G}$ are representatives of isomorphism classes in $\mathcal{P}^0_\mathbb{G}$. By homogeneity, that is, because partial automorphisms can be extended to full automorphisms, we can completely describe the $\Aut(\mathbb{G})$-orbit of $\mathcal{C}_\mathbb{G}\otimes \mathbb{K}$:  the orbit of $[C]\in\mathcal{C}_\mathbb{G}\otimes \mathbb{K}$ is $\{[D] \in \mathcal{P}^0_\mathbb{G} \mid K_0(D) \cong K_0(C)\}$, and the orbit of the whole set is all of $\mathcal{P}^0_\mathbb{G}$. Now we ask about the unital component of $\mathcal{P}_\mathbb{G}$: what is the $\Aut(\mathbb{G})$-orbit $\Aut(\mathbb{G})_{\mathcal{P}_\mathbb{G}}(\mathcal{C}_\mathbb{G})$ of $\mathcal{C}_\mathbb{G}$ in $\mathcal{P}^1_\mathbb{G} \subseteq \mathcal{P}_\mathbb{G}$? Again by homogeneity, this is really the same question that was addressed in Section~\ref{subsection:unit}, and the description of $\Aut(\mathbb{G})_{\mathcal{P}_\mathbb{G}}(\mathcal{C}_\mathbb{G})$ supported by the present article is probabilistic. For an idea of the probability that a suitably large randomly chosen element of $\mathcal{P}^1_\mathbb{G}$ lies in $\Aut(\mathbb{G})_{\mathcal{P}_\mathbb{G}}(\mathcal{C}_\mathbb{G})$, see Conjecture~\ref{conjecture:bernoulli}.
\end{remark}

\section{Random graph models} \label{section:models}

We will consider three methods of constructing directed multigraphs on $n$ labelled vertices by randomly adding edges. Equivalently, we create adjacency matrices $A(n)$ whose entries are random variables taking values in the nonnegative integers.

\begin{enumerate}[1.]
\item First, we insist that these entries be independent, leading in particular to \emph{Bernoulli digraphs} $\dd_{n,q}$ in which each possible edge (allowing loops but not multiple edges) occurs independently with fixed probability $q$. That is, each entry of $A(n)$ is either $1$ (with probability $q$) or $0$ (with probability $1-q$).
\item By contrast, there are two sources of dependence among the entries of our second model of random adjacency matrix, which is now required to be symmetric with each row summing to a fixed integer $r\ge3$. That is, we build \emph{random $r$-regular undirected multigraphs} and convert them into directed ones by replacing each edge (between given vertices $i,j$) by two (one from $i$ to $j$ and one from $j$ to $i$). Specifically, we use the \emph{perfect matchings model}
\[
\g_{n,r}=\underbrace{\g_{n,1}+\dots+\g_{n,1}}_r,
\]
that is, the union of $r$ independent, uniformly random perfect matchings on the $n\in2\nn$ vertices.
\item The third model $\ee_{n,q}$ is in a sense intermediate between the first two. These \emph{Erd\H{o}s--R\'{e}nyi graphs} are constructed in the same manner as $\dd_{n,q}$, but symmetrically. They exhibit a pattern of asymptotic connectivity similar to $\dd_{n,q}$, and asymptotic $K$-theoretic  behaviour similar to $\g_{n,r}$.
\end{enumerate}

In each case, we make statements about the asymptotic probabilities of:
\begin{itemize}
\item strong connectedness of the graph (entailing simplicity and pure infiniteness of the graph algebra);
\item nonsingularity of $M(n)=A(n)^t-I$ (leading to $K_1=0$ for the graph algebra);
\item occurrences of specific Sylow $p$-subgroups of the $K_0$-group of the graph algebra, for fixed primes $p$;
\item cyclicity of (Sylow subgroups of) the $K_0$-group (allowing us to compute or estimate the probability that the graph algebra is stably isomorphic to a Cuntz algebra).
\end{itemize}

We also discuss situations in which we can compute or estimate the asymptotic probabilities of the associated edge shift being flow equivalent to a full shift (Section~\ref{subsection:flow}) or the graph $\cs$-algebra being \emph{exactly} isomorphic to a Cuntz polygon or Cuntz algebra (Section~\ref{subsection:exact}), and investigate the scope for analysing related distributions (Section~\ref{subsection:contiguity}).

\begin{notation} \label{notation}
If $G$ is a finite group and $p$ a prime, then $G_p$ denotes the Sylow $p$-subgroup of $G$. If $P$ is a finite set of primes, then a \emph{$P$-group} is a group whose order is a product of powers of primes in $P$, and we write $G_P$ for $\bigoplus_{p\in P}G_p$. We use the symbol $\pp$ for probabilities. A sequence of events $\mathscr{E}_n$, $n\in\nn$, is said to occur \emph{asymptotically almost surely}, or \emph{with high probability}, if $\lim_{n\to\infty}\pp(\mathscr{E}_n)=1$. 
\end{notation}
 
\subsection{Bernoulli digraphs} \label{subsection:bernoulli}

The following in particular applies to Bernoulli digraphs $\dd_{n,q}$. Note that, since the entries of the adjacency matrices of these graphs are all either $0$ or $1$, if $q\in(0,1)$ then the condition \eqref{eqn:balance-bernoulli} holds for all primes $p$.

\begin{theorem} \label{thm:bernoulli}
Fix a constant $\varepsilon\in(0,1)$ and a finite set $S$ of nonnegative integers. For each $n\in\nn$, let $A(n)$ be an $n\times n$ matrix with independent random entries taking values in $S$ such that, for every $i,j\in\{1,\dots,n\}$ and every prime $p$,
\begin{equation} \label{eqn:balance-bernoulli}
\max_{s\in S}\pp\left(A(n)_{ij} \equiv s \mod p\right)\le 1-\varepsilon.
\end{equation}
Then:
\begin{enumerate}[1.]
\item \label{ber1} the graph $\hh_n$ associated with $A(n)$ is asymptotically almost surely strongly connected;
\item \label{ber2} $M(n)=A(n)^t-I$ is asymptotically almost surely nonsingular (and so is $A(n)$ itself);
\item \label{ber3} $\cs(\hh_n)$ is asymptotically almost surely stably isomorphic to a Cuntz polygon;
\item \label{ber4} for any finite set $P$ of primes and any finite abelian $P$-group $G$,
\begin{equation} \label{eqn:ersylow}
\lim_{n\to\infty} \pp\left(K_0(\cs(\hh_n))_P \cong G\right) = \frac{1}{|\Aut(G)|} \prod_{p\in P} \prod_{k=1}^\infty\left(1-p^{-k}\right);
\end{equation}
\item \label{ber5} assuming that the entries of $A(n)$ are identically distributed, the probability $c_n$ that $\cs(\hh_{n})$ is stably isomorphic to a Cuntz algebra satisfies
\begin{equation} \label{eqn:ercuntz}
\lim_{n\to\infty} c_n = \prod_{p\text{ prime}}\left(1+\frac{1}{p^2-p}\right) \prod_{k=2}^\infty\zeta(k)^{-1} \approx 0.84694,
\end{equation}
where $\zeta$ is the Riemann zeta function.
\end{enumerate}
\end{theorem}

\begin{proof}
\begin{enumerate}[1.]
\item For the first assertion, see \cite[Theorem13.9]{Frieze:2016tw}. While this theorem is specifically about loopless Bernoulli digraphs, adding loops or multiple edges does not affect connectivity.

\item The second assertion is a consequence of \cite[Corollary 3.3]{Bourgain:2010aa}; there, $S$ can be any finite set of complex numbers (with \eqref{eqn:balance-bernoulli} replaced by $\max_{s\in S}\pp\left(A(n)_{ij} = s\right)\le 1-\varepsilon$), so in particular, the entries of $M(n)$ also satisfy the necessary hypotheses. (In fact, it is shown there that the asymptotic probability of singularity is $(\sqrt{1-\varepsilon}+o(1))^n$.)

\item Since $A(n)$ with high probability is not a permutation matrix, and we have just established strong connectivity of $\hh_n$ and nonsingularity of $M(n)$, the third claim follows from Proposition~\ref{prop:kpoly}.\ref{poly3}.

\item For the computation of the distribution of Sylow $p$-subgroups, see \cite[Corollary 3.4]{Wood:2019wn}.

\item For the final equation, note that by definition and Theorem~\ref{eqn:cpk}.\ref{poly2}, and in light of Theorem~\ref{thm:bernoulli}.\ref{ber3}, $c_n$ is asymptotically equal to the probability that $\coker M(n) \cong K_0(\cs(\hh_n))$ is cyclic. Lest the appearance of the zeta function in the limiting value of $c_n$ seem overly mysterious, consider the following derivation, which is heuristic only in skipping the justification of exchanging the limits $n\to \infty$ and $|P|\to \infty$. Let us define the related quantity
\[
	c^P_n = \pp\left(K_0(\cs(\hh_n))_P \text{ is cyclic}\right) ,
\]
so that $\lim_{n\to \infty} c_n = \lim_{n\to \infty} \lim_{|P|\to\infty} c^P_n$. If we could exchange the two limits, then we could compute this by taking the $|P|\to \infty$ limit of~\eqref{eqn:ersylow}. First, recall that for any prime $p$ and any integer $N\ge1$
\[
|\Aut(\zz/p^N\zz)| = p^N(1-p^{-1}) ,
\]
as well as $|\Aut(G_1 \oplus G_2)| = |\Aut(G_1)| \cdot |\Aut(G_2)|$ when $|G_1|$ and $|G_2|$ are coprime. Let $N_P = (N_p)_{p\in P}$ denote a sequence of non-negative integers, and recall also that a cyclic group of order divisible only by some power of $p\in P$ must be of the form $\bigoplus_{p\in P} \zz/p^{N_p}\zz$, so that
\begin{align*}
	\lim_{n\to\infty} c_n^P
	&= \lim_{n\to\infty} \pp\left(K_0(\cs(\hh_n))_P \text{ is cyclic}\right) \\
	&= \lim_{n\to\infty} \sum_{N_P} \pp\left(K_0(\cs(\hh_n))_P \cong \bigoplus_{p\in P} \zz/p^{N_p}\zz\right) \\
	&\text{``=''} \sum_{N_P} \lim_{n\to\infty} \pp\left(K_0(\cs(\hh_n))_P \cong \bigoplus_{p\in P} \zz/p^{N_p}\zz\right) \\
	&= \sum_{N_P} \frac{1}{|\Aut\left(\bigoplus_{p\in P} \zz/p^{N_p}\zz\right)|} \prod_{p\in P} \prod_{k=1}^\infty \left(1-p^{-k}\right) \\
	&= \sum_{N_P} \prod_{p\in P} \frac{1}{|\Aut\left(\zz/p^{N_p}\zz\right)|} \prod_{k=1}^\infty \left(1-p^{-k}\right) \\
	&= \prod_{p\in P} \left(1 + \sum_{N=1}^\infty \frac{p^{-N}}{(1-p^{-1})}\right) \prod_{k=1}^\infty \left(1-p^{-k}\right) \\
	&= \prod_{p\in P} \left(1-\frac{1}{p}+\frac{1}{p-1}\right) \prod_{k=2}^\infty\left(1-p^{-k}\right)\\
	&= \left[\prod_{p\in P} \left(1+\frac{1}{p^2-p}\right)\right]
		\left[\prod_{k=2}^\infty \prod_{p\in P}\left(1-p^{-k}\right)\right] .
\end{align*}
By taking the limit $|P|\to\infty$ and using Euler's product formula
\[
\zeta(k)^{-1} = \prod_{p \text{ prime}}\left(1-p^{-k}\right),
\]
we recover~\eqref{eqn:ercuntz}.
Of course, in this heuristic derivation we have not justified the exchanges of the sum over $N_P$ and the $n\to\infty$, $|P|\to\infty$ limits and that is part of what is accomplished in \cite[Theorem 1.2]{Nguyen:2022uf} (with $u=0$ and $\alpha_n=\varepsilon$). On the other hand, while this theorem essentially does establish \eqref{eqn:ercuntz}, we cannot apply it automatically because the entries of $M(n)=A(n)^t-I$ are not identically distributed (the probability distribution for the diagonal entries has been shifted by $1$). However, a close inspection of its proof reveals that the conclusion remains valid. We now offer a brief justification of this claim.

One must show that, for every prime $p$, $\rank(M(n) \text{ mod }p) \ge n-1$. `Small', `medium' and `large' primes are considered separately. The case of small primes \cite[Proposition 2.1]{Nguyen:2022uf} is already observed for random matrices whose entries are not necessarily identical. The medium primes case \cite[Proposition 2.2]{Nguyen:2022uf} depends on Odlyzko's lemma \cite[Lemma 3.1]{Nguyen:2022uf} and its `union bound' corollary \cite[Corollary 3.2]{Nguyen:2022uf} (whose proofs are easily seen to cover random vectors with non-identical entries that share the same bound, which is certainly the case for us), \cite[Theorem 5.3]{Nguyen:2022uf} (which does not assume identical distributions), and \cite[Theorem 5.2]{Nguyen:2022uf} (the $u=0$ case of which is \cite[Theorem A.1]{Nguyen:2020aa}, discussed below). The large primes case \cite[Proposition 2.3]{Nguyen:2022uf}, our use of which falls under \cite[\S 6.1]{Nguyen:2022uf}, depends on the Erd\H{o}s--Littlewood--Offord result \cite[Theorem 6.3]{Nguyen:2022uf} (which by convexity, that is, conditioning over the possible values taken by the first entry of the random vector $X$ and using the triangle inequality, implies the version in which this entry is not identical to the others), \cite[Lemma 6.4]{Nguyen:2022uf} (which follows from Odlyzko's lemma), \cite[Lemma 6.5]{Nguyen:2022uf} (which is deterministic) and \cite[Lemma 6.2]{Nguyen:2022uf} (which follows from the previous three results).

It remains to carefully inspect the proof of \cite[Theorem A.1]{Nguyen:2020aa}, which is split into many lemmas, propositions and theorems. The key ones for us to examine are \cite[Theorem A.15]{Nguyen:2020aa} (which is the aforementioned Erd\H{o}s--Littlewood--Offord result), \cite[Lemma A.9]{Nguyen:2020aa}, \cite[Proposition A.18]{Nguyen:2020aa} and \cite[Lemma A.11]{Nguyen:2020aa}. The other results are either deterministic or are proved from these key ones (together with Odlyzko's lemma) without further appeal to the assumption of identically distributed entries.

For each of these key results, the only change is that instead of one measure $\mu$ specifying the distribution of the entries of the random vector $X$, there are two, $\mu$ and $\mu_1$ say (to account for the shifted first entry), and products of the form $\prod_{l=1}^n\widehat\mu(t_l)$ should be replaced by $\widehat\mu_1(t_1)\prod_{l=2}^n\widehat\mu(t_l)$. Here, $\widehat{\mu}$ denotes the Fourier transform defined on \cite[p.\ 287]{Nguyen:2020aa}, the absolute value of which is not affected by the shift $\mu_1(t)=\mu(t+1)$. So, the proof of \cite[Lemma A.9]{Nguyen:2020aa} carries over unchanged. For \cite[Proposition A.18]{Nguyen:2020aa} (which is used to prove \cite[Lemma A.11]{Nguyen:2020aa}), the measure $\nu_1$ associated to $\mu_1$ (obtained from \cite[Proposition 3.6]{Maples:2013aa}) has the same Fourier transform as $\nu$ associated to $\mu$ (since by construction, $\widehat\nu(\xi)=1-\gamma+\gamma|\widehat\mu(\xi)|^2$ for a suitable $\gamma\in(0,1)$), so again the same proof works (in fact with $\nu_1=\nu$). The rest of the proof of \cite[Lemma A.11]{Nguyen:2020aa} also needs no change.
\qedhere
\end{enumerate}
\end{proof}

\begin{remark}  \label{remark:logn}
Apart from expanding the analysis from finitely many primes to all primes simultaneously, a substantial portion of Nguyen and Wood's work in \cite{Nguyen:2022uf} goes towards proving that, for random matrices with identically distributed entries, the probability bound $1-\varepsilon$ can be weakened to $1-\alpha_n$ with $\alpha_n\ge n^{-1+\varepsilon}$. The argument we have described in the proof of Theorem~\ref{thm:bernoulli} works for $\alpha_n\ge n^{-1/6+\varepsilon}$ (which is the case considered in \cite[\S 6.1]{Nguyen:2022uf}), but a great deal more care is needed to extend to the general case. On the other hand, as pointed out in \cite{Nguyen:2022uf}, the bound $\alpha_n\ge n^{-1+\varepsilon}$ is asymptotically best possible: if the matrix entries can take the value $0$ with probability at least $1-\log n/n$, then with nonzero probability the matrix has a row of zeros, so is singular and in particular has infinite cokernel. The statistics of these sparser graphs are thus rather different, a fact that is supported by the data (see Table~\ref{table:dlog}).
\end{remark}

\subsection{Random regular multigraphs} \label{subsection:regular}

For a finite abelian group $G$, a symmetric, $\zz$-bilinear map $\varphi\colon G\times G\to\cc^*$ is a \emph{perfect pairing} if the only $g\in G$ with $\varphi(g,G)=1$ is $g=0$. (Equivalently, $\varphi$ induces an isomorphism $G\to \widehat{G} = \Hom(G,\cc^*)$ via $g\mapsto \varphi(g,\cdot)$.) We define
\begin{equation} \label{eqn:ng}
N(G) = \frac{|\{\text{symmetric, bilinear, perfect } \varphi\colon G\times G\to\cc^*\}|}{|G|\cdot|\Aut(G)|}.
\end{equation}
If $G$ is a finite abelian $p$-group
\[
G=\bigoplus_{i=1}^M \zz/p^{\lambda_i}\zz
\]
with $\lambda_1\ge\lambda_2\ge\dots\ge\lambda_M$, then
\begin{equation} \label{eqn:ngp}
N(G) = p^{-\sum_i\frac{\mu_i(\mu_i+1)}{2}}\prod_{i=1}^{\lambda_1}\prod_{j=1}^{\lfloor\frac{\mu_i-\mu_{i+1}}{2}\rfloor}(1-p^{-2j})^{-1},
\end{equation}
where $\mu_i=|\{j\mid \lambda_j\ge i\}|$ (see \cite[\S1]{Wood:2017tk}). It is worth noticing that, if the $\lambda_i$ are interpreted as the lengths of the rows of a Young diagram, the $\mu_j$ are the lengths of its columns.

The expression \eqref{eqn:kd} makes use of the following observation. If $G_1$ is a finite abelian $p_1$-group and $G_2$ is a finite abelian $p_2$-group with $p_2\ne p_1$, then $N(G) = N(G_1)\cdot N(G_2)$ for $G=G_1\oplus G_2$. After all, if we count the perfect pairings in the definition~\eqref{eqn:ng} of $N(G)$ as group homomorphisms $G\to \widehat{G}$, any such homomorphism comes from an independent pair of homomorphisms $G_1 \to \widehat{G}_1$ and $G_2 \to \widehat{G}_2$, which are themselves perfect pairings. The factorisation of $N(G)$ then follows by noting also the obvious factorisations $|G|=|G_1| \cdot |G_2|$ and $|\Aut(G)| = |\Aut(G_1)| \cdot |\Aut(G_2)|$.

The following adaptation of \cite{Nguyen:2018vh} was mostly observed in \cite{Jacelon:2023aa}.

\begin{theorem} \label{thm:regular}
Let $n\in2\nn$ and $3\le r\in\nn$, and recall that $\g_{n,r}$ denotes a random $r$-regular multigraph on $n$ vertices. Then:
\begin{enumerate}[1.]
\item \label{reg1} $\g_{n,r}$ is asymptotically almost surely strongly connected;
\item \label{reg2} $M(n)=A(n)^t-I$ is asymptotically almost surely nonsingular (and so is $A(n)$);
\item \label{reg3} $\cs(\g_{n,r})$ is asymptotically almost surely stably isomorphic to a Cuntz polygon;
\item \label{reg4} for any finite set $P$ of odd primes not dividing $r-1$, and any finite abelian $P$-group $G$,
\begin{equation} \label{eqn:kd}
\lim_{n\in2\nn}\pp\left(K_0(\cs(\g_{n,r}))_P\cong G\right) = \prod_{p\in P} N(G_p)\prod_{k=1}^\infty\left(1-p^{-2k+1}\right).
\end{equation}
\end{enumerate}
\end{theorem}

\begin{proof}
The first and last assertions appear in \cite[Theorem 6.3]{Jacelon:2023aa}. The second then follows from \cite[Corollary 4.2]{Nguyen:2018vh} (see also the rest of the proof of \cite[Theorem 1.6]{Nguyen:2018vh}), and the third then follows from Proposition~\ref{prop:kpoly}.\ref{poly3}, since $A(n)$ is not a permutation matrix (because each of its rows sums to $r$) and we have already established asymptotic strong connectivity of $E$ as well as asymptotic nonsingularity of $A(n)^t-I$.
\end{proof}
 
In particular, for any odd prime $p$ coprime to $r-1$ and any integer $N\ge0$,
\begin{align} \label{eqn:ncyclic}
\lim_{n\in2\nn}\pp\left(K_0(\cs(\g_{n,r}))_p\cong \zz/p^N\zz\right) &= p^{-N}\prod_{k=1}^\infty\left(1-p^{-2k+1}\right)\\
&\approx p^{-N}, \nonumber
\end{align}
while
\begin{align} \label{eqn:cyclicn}
\lim_{n\in2\nn}\pp\left(K_0(\cs(\g_{n,r}))_p\cong (\zz/p\zz)^N\right) &= p^{-\frac{N(N+1)}{2}}\prod_{j=1}^{\lfloor N/2 \rfloor}(1-p^{-2j})^{-1} \prod_{k=1}^\infty\left(1-p^{-2k+1}\right)\\
&\approx p^{-\frac{N(N+1)}{2}}, \nonumber
\end{align}
the approximations holding for sufficiently large primes $p$.

Using \eqref{eqn:kd} and \eqref{eqn:ncyclic}, we can compute the asymptotic probability that $(K_0(\cs(\g_{n,r}))_P$ is cyclic, for any \emph{finite} set $P$ of odd primes not dividing $r-1$. The following also applies to Erd\H{o}s--R\'{e}nyi graphs $\ee_{n,q}$ for $P=\{p\}$, where $p$ can be any prime (see Theorem~\ref{thm:erdos}).

\begin{proposition} \label{prop:pcyclic}
Let $(\ee_n)_{n\in\nn}$ be a family of random graphs and $P$ a finite set of primes such that
\begin{equation} \label{eqn:kdgeneral}
\lim_{n\in\nn}\pp\left(K_0(\cs(\ee_n))_P\cong G\right) = \prod_{p\in P} N(G_p)\prod_{k=1}^\infty\left(1-p^{-2k+1}\right)
\end{equation}
for every finite abelian $P$-group $G$. Then,
\begin{equation} \label{eqn:allcyclic}
\lim_{n\in\nn}\pp\left(K_0(\cs(\ee_n))_P \text{ is cyclic}\right) = \prod_{p\in P} \prod_{k=2}^\infty\left(1-p^{-2k+1}\right).
\end{equation}
\end{proposition}

\begin{proof}
As in the proof of Theorem~\ref{thm:bernoulli}, let us write $c_n^P:=\pp\left(K_0(\cs(\ee_n))_P \text{ is cyclic}\right)$, and for $N_P = (N_p)_{p\in P}$ let $c_n^{N_P} = \pp\left(K_0(\cs(\ee_n))_P \cong \bigoplus_{p\in P} \zz/p^{N_p}\zz\right)$, so that $c_n^P = \sum_{N_P} c_n^{N_P}$.
To show the equality~\eqref{eqn:allcyclic} we essentially want to interchange the limit with the summation to get
\begin{equation} \label{eqn:c_n-vitali}
	\lim_{n\in\nn} c_n^P
	= \lim_{n\in\nn} \sum_{N_p} c_n^{N_P}
	= \sum_{N_P} \lim_{n\in\nn} c_n^{N_P}
\end{equation}
and sum the right-hand side, where by hypothesis~\eqref{eqn:kdgeneral} and~\eqref{eqn:ncyclic} we have
\begin{align*}
	\sum_{N_P} \lim_{n\in\nn} c_n^{N_P}
	&= \sum_{N_P} \prod_{p\in P} p^{-N_p} \prod_{k=1}^\infty \left(1-p^{-2k+1}\right) \\
	&= \prod_{p\in P} \sum_{N_p=0}^\infty p^{-N_p} \prod_{k=1}^\infty \left(1-p^{-2k+1}\right) \\
	&= \prod_{p\in P} \prod_{k=2}^\infty \left(1-p^{-2k+1}\right) .
\end{align*}

The interchange in~\eqref{eqn:c_n-vitali} is justified by the Vitali convergence theorem \cite[Theorem III.6.15]{Dunford:1958}, which in general applies to the interchange of a limit with integration of absolutely summable functions over a measure space. In our case, with the pointwise convergence of $c_n^{N_P}$ over the discrete measure space of all $N_P$'s, the remaining hypothesis to check is \emph{equismallness at infinity} (a.k.a\ \emph{equisummability at infinity}; for series this is spelled out in~\cite[Theorem 44.2]{Treves:1967}). Namely, for every $\varepsilon>0$ we need to find a finite set $X_\varepsilon$ such that for all $n$
\[
	\sum_{N_P \not\in X_\varepsilon} c_n^{N_P} < \varepsilon .
\]
By virtue of being probability distributions (hence absolutely summable, with no need to take absolute values because of pointwise positivity) there exist finite sets $Z_\varepsilon$ and $Z_{n,\varepsilon}$ of abelian $P$-groups such that
\[
	\sum_{G\not\in Z_\varepsilon} \lim_{m\in\nn}\pp\left(K_0(\cs(\ee_m))_P \cong G\right) < \frac{\varepsilon}{2}
	\quad\text{and}\quad
	\sum_{G\not\in Z_{n,\varepsilon}} \pp\left(K_0(\cs(\ee_n))_P \cong G\right) < \varepsilon .
\]
By pointwise convergence of the probability distributions in~\eqref{eqn:kdgeneral}, there also exists an $M\in\nn$ such that for all $n\ge M$ and $G\in Z_{\varepsilon}$
\[
\left|\pp\left(K_0(\cs(\ee_n))_P \cong G\right) - \lim_{m\in\nn}\pp\left(K_0(\cs(\ee_m))_P \cong G\right) \right| < \frac{\varepsilon}{2|Z_\varepsilon|} .
\]
Hence, we can set $Z_{n,\varepsilon} = Z_\varepsilon$ for all $n\ge M$. Noting that $c^{N_P}$ and $c_n^{N_P}$ are restrictions of the above probability distributions to cyclic groups and setting $X_\varepsilon$ to be the intersection of the set $Z_\varepsilon \cup \bigcup_{n<M} Z_{n,\varepsilon}$ with cyclic groups shows that $c_n^{N_P}$ is equismall at infinity, allowing us to apply Vitali's convergence theorem in~\eqref{eqn:c_n-vitali}.
\end{proof}

\begin{remark} \label{remark:regrem}
If $r=2^j+1$ for some $j$, then \eqref{eqn:ncyclic}, \eqref{eqn:cyclicn} and \eqref{eqn:allcyclic} hold for $\g_{n,r}$ for any finite set $P$ of odd primes (as all odd primes are coprime to $r-1=2^j$). This is why we we pay special attention to these values of $r$ in Section~\ref{section:stats}. But we do not have any theoretical distribution for $p=2$. In addition, because of the subtleties arising from the two limits, $n\to\infty$ and enlarging $P$ to the set of all primes, one cannot immediately extend \eqref{eqn:allcyclic} to infinitely many primes. For Bernoulli digraphs, this more delicate analysis is carried out in \cite{Nguyen:2022uf}, but there is as yet no such result for $r$-regular multigraphs. These are exactly the obstacles to computing the asymptotic probability that $\cs(\g_{n,r})$ is stably isomorphic to a Cuntz algebra, which we mentioned in the Introduction.
\end{remark}

\subsection{Erd\H{o}s--R\'{e}nyi graphs} \label{subsection:erdosrenyi}

The following in particular applies to Erd\H{o}s--R\'{e}nyi graphs $\ee_{n,q}$, the symmetric versions of $\dd_{n,q}$, provided that $q\in(0,1)$.

\begin{theorem} \label{thm:erdos}
Fix a constant $\varepsilon\in(0,1)$ and a finite set $S$ of nonnegative integers. For each $n\in\nn$, let $A(n)$ be a symmetric $n\times n$ matrix such that, for every $i\le j\in\{1,\dots,n\}$, the entries $A(n)_{ij}$ are independent random variables taking values in $S$ that satisfy
\begin{equation} \label{eqn:balance-erdos}
\max_{s\in S}\pp\left(A(n)_{ij} \equiv s \mod p\right)\le 1-\varepsilon
\end{equation}
for every prime $p$. Then:
\begin{enumerate}[1.]
\item \label{er1} the graph $\ee_n$ associated with $A(n)$ is asymptotically almost surely strongly connected;
\item \label{er2} $M(n)=A(n)^t-I$ is asymptotically almost surely nonsingular (and so is $A(n)$ itself);
\item \label{er3} $\cs(\ee_n)$ is asymptotically almost surely stably isomorphic to a Cuntz polygon;
\item \label{er4} for any prime $p$ and any finite abelian $p$-group $G$,
\begin{equation} \label{eqn:kder}
\lim_{n\in\nn}\pp\left(K_0(\cs(\ee_n))_p\cong G\right) = N(G)\prod_{k=1}^\infty\left(1-p^{-2k+1}\right),
\end{equation}
where $N(G)$ is as in \eqref{eqn:ng}, \eqref{eqn:ngp}.
\end{enumerate}
\end{theorem}

\begin{proof}
The first assertion follows from \cite{Erdos:1959aa} (see \cite[Theorem 4.1]{Frieze:2016tw}) and the last from \cite[Theorem 3.12]{Wood:2023aa}. The second then follows from \cite[Corollary 4.2]{Nguyen:2018vh}, and  the third then follows from Proposition~\ref{prop:kpoly}, exactly as in the proof of Theorem~\ref{thm:bernoulli}.\ref{ber3}.
\end{proof}

By Theorem~\ref{thm:erdos}.\ref{er4}, \emph{all} Sylow $p$-subgroups of $K_0(\cs(\ee_{n,q}))$, including $p=2$, follow the limiting behaviour \eqref{eqn:ncyclic}, \eqref{eqn:cyclicn} and \eqref{eqn:allcyclic} with $P=\{p\}$. The data we have collected for these graphs are consistent with this (see Figure~\ref{fig:bargraphsymer2} and Table~\ref{table:eq}).

\subsection{The signed Bowen--Franks group} \label{subsection:flow}

As discussed in the Introduction, every (directed multi-) graph $E$ has an associated edge shift $\sigma \colon \Sigma_E \to \Sigma_E$, whose flow equivalence class is determined by the signed Bowen--Franks group $\mathrm{BF}_+(E)=\sgn(\det(I-A_E))\coker(I-A_E)$ (assuming that $E$ is strongly connected and does not consist of a single cycle). For the random graph models discussed above, we are interested in the asymptotic probability of flow equivalence to a full shift  (whose Bowen--Frank invariant is $-\zz/k\zz$ for some $k$). If we make the heuristic assumption that positivity of the determinant is asymptotically independent of cyclicity of the cokernel (a not unreasonable notion given the data; see Tables~\ref{table:dnqextra},~\ref{table:dnqextra2},~\ref{table:enqextra},~\ref{table:gnrextra}), then we are left with the question: what is the shape of the distribution of $\det(I-A(n))$?

There is one situation in which symmetry is immediately apparent, namely, when the entries of $A(n)-I$ are identically distributed, simply because in this case swapping rows preserves the distribution. Moreover, our careful adaptation of \cite[Theorem 1.2]{Nguyen:2022uf} in the proof of Theorem~\ref{thm:bernoulli}.\ref{ber5} is now no longer necessary. A direct application this theorem tells us that the asymptotic probability of cyclicity is once again given by \eqref{eqn:ercuntz}.

In summary of the above discussion, we have the following. Recall that $\zeta$ denotes the Riemann zeta function.

\begin{theorem} \label{thm:bf}
Fix a constant $\varepsilon\in(0,1)$ and a finite set $S$ of nonnegative integers. For each $n\in\nn$, let $B(n)$ be an $n\times n$ matrix with independent, identically distributed random entries taking values in $S$ such that, for every $i,j\in\{1,\dots,n\}$ and every prime $p$,
\[
\max_{s\in S}\pp\left(B(n)_{ij} \equiv s \mod p\right)\le 1-\varepsilon.
\]
Let $\hh_n$ be the graph whose adjacency matrix is $A(n)=B(n)+I$. Then, the asymptotic probability that the edge shift $\sigma \colon \Sigma_{\hh_n} \to \Sigma_{\hh_n}$ is flow equivalent to a full shift is
\begin{equation} \label{eqn:fullshift}
\frac{1}{2}\prod_{p\text{ prime}}\left(1+\frac{1}{p^2-p}\right) \prod_{k=2}^\infty\zeta(k)^{-1}
\end{equation}
which, as a percentage, is about $42$.
\end{theorem}

Just as Theorem~\ref{thm:bernoulli} in particular applies to Bernoulli digraphs $\dd_{n,q}$, Theorem~\ref{thm:bf}  applies to \emph{shifted} Bernoulli digraphs $\dd_{n,q}+I$ (see Table~\ref{table:dnqextra2}). These graphs have either one loop (with probability $1-q$) or two (with probability $q$) at every vertex, with other edges distributed in the same way as for $\dd_{n,q}$. With randomness specified in this manner, the pithy interpretation of \eqref{eqn:fullshift} is that this is the asymptotic probability that a random subshift of finite type is flow equivalent to a full shift. Of course, it might be reasonable to expect that, asymptotically, shifting by the identity matrix should not have any effect on the distribution of the determinant. The data do support this expectation (see Table~\ref{table:dnqextra}), and we conjecture it to be true.

\begin{conjecture} \label{conjecture:fullshift}
For $q\in(0,1)$,
\begin{multline*}
\lim_{n\to\infty} \pp\left(\sigma \colon \Sigma_{\dd_{n,q}} \to \Sigma_{\dd_{n,q}} \text{ is flow equivalent to a full shift}\right) \\
 = \frac{1}{2}\prod_{p\text{ prime}}\left(1+\frac{1}{p^2-p}\right) \prod_{k=2}^\infty\zeta(k)^{-1}
  \approx 0.42347.
\end{multline*}
\end{conjecture}

For the symmetric graphs $\ee_{n,q}$, it is expected but not yet proved (though supported by the data; see Table~\ref{table:eq}), that Proposition~\ref{prop:pcyclic} should hold when $P$ is enlarged to the set of all primes. In other words, we expect that
\begin{equation} \label{eqn:ecyclic}
\lim_{n\to\infty} \pp\left(K_0(\cs(\ee_{n,q})) \text{ is cyclic}\right) = \prod_{p\text{ prime}} \prod_{k=2}^\infty\left(1-p^{-2k+1}\right) \approx 0.79352.
\end{equation}
It seems however that for $\ee_{n,q}$ (and also $\g_{n,r}$), the determinant may not be symmetrically distributed about the origin and its asymptotic behaviour is a little harder to predict (see Tables~\ref{table:enqextra},~\ref{table:gnrextra}). This means that we cannot simply bisect \eqref{eqn:ecyclic} (or for $\g_{n,r}$, \eqref{eqn:gnrcuntz}) and conjecture that value as the asymptotic full-shift probability.

\begin{question} \label{q1}
What is the asymptotic distribution of the signs of of $\det(I-A(n))$ when $A(n)$ is the adjacency matrix of $\ee_{n,q}$ or $\g_{n,r}$? Is it still the case that negativity of the determinant and cyclicity are asymptotically independent events? Namely, do the probabilities
\begin{equation} \label{eqn:enq-shiftprob}
\begin{aligned}
\delta_{n,q} &:=  \pp\left(\det(I-A(n))<0\right) \\
\sigma_{n,q} &:= \pp\left(\det(I-A(n))<0 \wedge \coker(I-A(n)) \text{ cyclic}\right)
\end{aligned}
\end{equation}
for $\ee_{n,q}$ ($q\in(0,1)$) and
the probabilities
\begin{equation} \label{eqn:gnr-shiftprob}
\begin{aligned}
\varepsilon_{n,r} &:=  \pp\left(\det(I-A(n))<0\right) \\
\tau_{n,r} &:= \pp\left(\det(I-A(n))<0 \wedge \coker(I-A(n)) \text{ cyclic}\right)
\end{aligned}
\end{equation}
for $\g_{n,r}$ ($r\ge3$) converge, and if so, what are their limits $\delta_q$, $\sigma_q$, $\varepsilon_r$, $\tau_r$?
\end{question}

\subsection{Exact isomorphism} \label{subsection:exact}

We described in Section~\ref{subsection:exact} how, practically speaking, we can check whether a graph algebra is exactly (rather than stably) isomorphic to a Cuntz polygon (or Cuntz algebra). In Section~\ref{section:stats}, we provide estimates of the asymptotic probabilities of these events for $\dd_{n,q}$ and $\ee_{n,q}$ (see Tables~\ref{table:dnqextra},~\ref{table:enqextra}). Here, we offer an explanation of the numbers that we see there.

\begin{conjecture} \label{conjecture:bernoulli}
For $q\in(0,1)$,
\begin{align}
\lim_{n\to\infty} \pp\left(\cs(\dd_{n,q}) \text{ is isomorphic to a Cuntz algebra}\right) &= \prod_{k=2}^\infty\zeta(k)^{-1} \notag\\
&\approx 0.43576, \label{eqn:dcuntz}\\
\lim_{n\to\infty} \pp\left(\cs(\ee_{n,q}) \text{ is isomorphic to a Cuntz algebra}\right) &= \prod_{p \text{ prime}} \left(1-\frac{1}{p^2}\right) \prod_{k=2}^\infty\left(1-p^{-2k+1}\right) \notag\\
&\approx 0.51451, \label{eqn:ecuntz}\\
\lim_{n\to\infty} \pp\left(\cs(\dd_{n,q}) \text{ is isomorphic to a Cuntz polygon}\right) &= \prod_{p \text{ prime}} \left(1+\frac{1}{p^2-p}\right)^{-1} \notag\\
& \approx 0.48240, \label{eqn:dexact}\\
\lim_{n\to\infty} \pp\left(\cs(\ee_{n,q}) \text{ is isomorphic to a Cuntz polygon}\right) &= \prod_{p \text{ prime}} \left(1-\frac{1}{p^2}\right) \notag \\
&\approx 0.60793. \label{eqn:eexact}
\end{align}
\end{conjecture}

\begin{proof}[Heuristic `proof']
The derivation is easy if we are willing to freely and recklessly use the heuristic principles of \emph{independence} (I) and \emph{uniformity} (U). More precisely, let us assume that if $\hh_n$ is a suitably random family of graphs, then the following hold asymptotically.
\begin{enumerate}
\item[(I)] The events $\mathscr{C}_n$: `$K_0(\cs(\hh_n))$ is cyclic' and $\mathscr{U}_n$: `$[1_{\cs(\hh_n)}]$ is in the same automorphism orbit in $K_0(\cs(\hh_n))$ as the unit of the corresponding Cuntz polygon' are independent.
\item[(U)] Conditioned on $\mathscr{C}_n$, $[1_{\cs(\hh_n)}]$ is uniformly distributed in $K_0(\cs(\hh_n))$.
\end{enumerate}
Suppose that $K_0(\cs(\hh_n)) \cong \bigoplus_{p\in P} \zz/p^{N_p}\zz=: G_{N_P}$ for some finite set $P$ of primes. The class of the unit of the corresponding Cuntz algebra is a generator of the cyclic group $G_{N_P}$. Its automorphism orbit is the set of elements of full order, of which there are $\prod_{p\in P}p^{N_p}(1-p^{-1}) = |\Aut(G_{N_P})|$ many. Applying (U) in this case, the conditional probability that $\cs(\hh_n)$ is isomorphic to $\oo_{1+|G_{N_P}|}$ is
\[
\frac{|\Aut(G_{N_P})|}{|G_{N_P}|} = \prod_{p \in P}(1-p^{-1}).
\]
Then, proceeding as in the heuristic argument presented in the proof Theorem~\ref{thm:bernoulli}.\ref{ber5} and making no attempt to justify any rearrangement of limits,
\begin{align*}
\lim_{n\to\infty}\pp\left(\cs(\hh_n)\cong\oo_k \text{ for some } k\right) &= \lim_{n\to\infty}\pp(\mathscr{C}_n \cap \mathscr{U}_n)  \\
&= \lim_{n\to\infty} \sum_{N} \frac{|\Aut(G_{N_P})|}{|G_{N_P}|} \pp(K_0(\cs(\hh_n)) \cong G_{N_P}) \\
&= \prod_{p \text{ prime}} \sum_{N} \frac{|\Aut(\zz/p^N\zz)|}{|\zz/p^N\zz|} \lim_{n\to\infty} \pp\left(K_0(\cs(\hh_n))_p \cong \zz/p^N\zz\right).
\end{align*}
Applying \eqref{eqn:ersylow} in the case $\hh_n=\dd_{n,q}$, we get 
\begin{align*}
\lim_{n\to\infty}\pp\left(\cs(\dd_{n,q})\cong\oo_k \text{ for some } k\right) &= 
\prod_{p \text{ prime}} \sum_{N} \frac{|\Aut(\zz/p^N\zz)|}{|\zz/p^N\zz|} \cdot \frac{1}{|\Aut(\zz/p^N\zz)|}\prod_{k=1}^\infty\left(1-p^{-k}\right) \\
&= \prod_{p \text{ prime}} \sum_{N} p^{-N}\prod_{k=1}^\infty\left(1-p^{-k}\right) \\
&= \prod_{p \text{ prime}} \prod_{k=2}^\infty\left(1-p^{-k}\right) \\
&= \prod_{k=2}^\infty\zeta(k)^{-1}.
\end{align*}
From this we deduce that
\begin{align*}
\lim_{n\to\infty} \pp\left(\cs(\dd_{n,q})\cong\ppm \text{ for some } \bar{m}\right) &= \lim_{n\to\infty} \pp(\mathscr{U}_n) \\
&= \lim_{n\to\infty} \frac{\pp(\mathscr{C}_n \cap \mathscr{U}_n)}{\pp(\mathscr{C}_n)} &&\text{(appealing to (I))} \\
&= \prod_{p \text{ prime}} \left(1+\frac{1}{p^2-p}\right)^{-1} &&\text{(by \eqref{eqn:dcuntz} and \eqref{eqn:ercuntz})}.
\end{align*}
For $\hh_n=\ee_{n,q}$, recall that the normalisation factor $N(G_P)$ defined in \eqref{eqn:ngp} and appearing in \eqref{eqn:kder} simplifies to $N(\zz/p^N\zz)=\prod_{p \text{ prime}}p^{-N}=|G_P|^{-1}$ for the cyclic group $G_P$ (as in \eqref{eqn:ncyclic}). Assuming that \eqref{eqn:kder} behaves well when passing to the set of all primes, we have
\begin{align*}
\lim_{n\to\infty}\pp\left(\cs(\ee_{n,q})\cong\oo_k \text{ for some } k\right) &= 
\prod_{p \text{ prime}} \sum_{N} \frac{|\Aut(\zz/p^N\zz)|}{|\zz/p^N\zz|^2} \prod_{k=1}^\infty\left(1-p^{-2k+1}\right)\\
&= \prod_{p \text{ prime}} \left(1+\sum_{N=1}^\infty \frac{p^N(1-p^{-1})}{p^{2N}}\right) \prod_{k=1}^\infty\left(1-p^{-2k+1}\right) \\
&= \prod_{p \text{ prime}} \left(1+\frac{1}{p}\right) \prod_{k=1}^\infty\left(1-p^{-2k+1}\right) \\
&= \prod_{p \text{ prime}} \left(1-\frac{1}{p^2}\right) \prod_{k=2}^\infty\left(1-p^{-2k+1}\right).
\end{align*}
By the same token, assuming \eqref{eqn:ecyclic} as in the discussion following Conjecture~\ref{conjecture:fullshift}, we have
\begin{equation} \label{eqn:ecn}
\lim_{n\to\infty} \pp(\mathscr{C}_n) =  \prod_{p \text{ prime}}\prod_{k=2}^\infty\left(1-p^{-2k+1}\right)
\end{equation}
and then
\begin{align*}
\lim_{n\to\infty} \pp\left(\cs(\ee_{n,q})\cong\ppm \text{ for some } \bar{m}\right) &= \lim_{n\to\infty} \frac{\pp(\mathscr{C}_n \cap \mathscr{U}_n)}{\pp(\mathscr{C}_n)} &&\text{(by (I))}\\
&= \prod_{p \text{ prime}} \left(1-\frac{1}{p^2}\right) &&\text{(by \eqref{eqn:ecuntz} and \eqref{eqn:ecn})}. \qedhere
\end{align*}
\end{proof}

By contrast, it appears to be \emph{very} rare for $\cs(\g_{n,r})$ to be isomorphic to a Cuntz polygon. Indeed, for the values of $n$ and $r$ for which we collected data, this event essentially never happened (see Table~\ref{table:gnrextra}). But it is not an impossibility. The following example was generated by our computer code.

\begin{example} \label{example:regular}
Let $E$ be the $8$-regular graph on $6$ vertices whose adjacency matrix is
\[
A =
\begin{pmatrix}
	0 & 3 & 3 & 1 & 0 & 1 \\
	3 & 0 & 0 & 2 & 3 & 0 \\
	3 & 0 & 0 & 2 & 1 & 2 \\
	1 & 2 & 2 & 0 & 1 & 2 \\
	0 & 3 & 1 & 1 & 0 & 3 \\
	1 & 0 & 2 & 2 & 3 & 0
\end{pmatrix}.
\]
Since $A-I$ is not a permuted diagonal matrix, its underlying graph is obviously not that of a Cuntz polygon itself. The SNF factorisation $UMV=D$ of $M=A-I$ is
\[
\left(\begin{smallmatrix}
	1 & -6 & -6 & -6 & -6 & -6 \\
	0 & 1 & 0 & 0 & 0 & 0 \\
	0 & 0 & 1 & 0 & 0 & 0 \\
	0 & 0 & 0 & 1 & 0 & 0 \\
	0 & 0 & 0 & 0 & 1 & 0 \\
	0 & 0 & 0 & 0 & 0 & 1
\end{smallmatrix}\right)
\left(\begin{smallmatrix}
	-1 & 3 & 3 & 1 & 0 & 1 \\
	3 & -1 & 0 & 2 & 3 & 0 \\
	3 & 0 & -1 & 2 & 1 & 2 \\
	1 & 2 & 2 & -1 & 1 & 2 \\
	0 & 3 & 1 & 1 & -1 & 3 \\
	1 & 0 & 2 & 2 & 3 & -1
\end{smallmatrix}\right)
\left(\begin{smallmatrix}
	13 & 13 & 11 & 11 & 10 & 9 \\
	13 & 19 & 6 & 10 & 13 & 6 \\
	-1 & -6 & 3 & 0 & -3 & 2 \\
	-1 & -2 & 0 & -1 & -1 & 0 \\
	-8 & -5 & -9 & -7 & -5 & -7 \\
	-15 & -18 & -10 & -12 & -13 & -9
\end{smallmatrix}\right)
=
\left(\begin{smallmatrix}
	7 & 0 & 0 & 0 & 0 & 0 \\
	0 & 1 & 0 & 0 & 0 & 0 \\
	0 & 0 & 1 & 0 & 0 & 0 \\
	0 & 0 & 0 & 1 & 0 & 0 \\
	0 & 0 & 0 & 0 & 1 & 0 \\
	0 & 0 & 0 & 0 & 0 & 1
\end{smallmatrix}\right),
\]
from which we can deduce that 
\[
(K_0(\cs(E),[1_{\cs(E)}],K_1(\cs(E))) \cong (\zz/7\zz,-29,0) \cong (\zz/7\zz,1,0),
\]
since $29$ and $7$ are coprime. The conclusion is that $\cs(E)$ is isomorphic to $\oo_8$.
\end{example}

\begin{question} \label{q2}
Is it true that, for $r\ge 3$,  $\lim_{n\to\infty}\pp\left(\cs(\g_{n,r}) \cong\ppm \text{ for some } \bar{m}\right) = 0$?
\end{question}

\subsection{Contiguity} \label{subsection:contiguity}
By definition, events that occur asymptotically almost surely according to one random graph model also occur asymptotically almost surely for any \emph{contiguous} model. For example, the perfect matchings model $\g_{n,r}$ analysed in Section~\ref{subsection:regular} is contiguous to the \emph{uniform model} $\mathbb{G}'_{n,r}$, that is, a random element of the set of $r$-regular multigraphs with the uniform distribution, conditioned on there being no loops (see the discussion and references in \cite[Chapter 21]{Frieze:2016tw}). Had we adopted this model instead, the resultant graph algebras would therefore still be stably isomorphic to Cuntz polygons with high probability (see Theorem~\ref{thm:regular}). As for the `Erd\H{o}s--R\'{e}nyi' graphs $\ee_{n,q}$ of Section~\ref{subsection:erdosrenyi}, a better name for them might have been `Gilbert' after his analysis of this model in \cite{Gilbert:1959aa}. Moreover, the model $\mathbb{E}_{n,m}$ actually employed in \cite{Erdos:1959aa} is a different one, specifically the uniform distribution on the set of undirected graphs with $m$ edges on $n$ labelled vertices (where $n\in\nn$ and $0\le m\le \binom{n}{2}$). To borrow a phrase from \cite{Frieze:2016tw}, these two models $\ee_{n,q}$ and $\ee_{n,m}$ are `almost contiguous' in a quantifiably precise sense. This principle, when adapted for digraphs and combined with the asymptotic probabilities of singularity and strong connectedness in this setting, leads to the observation (Theorem~\ref{thm:uniform}) that \emph{most} of the $\cs$-algebras associated to digraphs on large vertex sets with suitably many edges are stably isomorphic to Cuntz polygons.

\begin{definition} \label{def:dnm}
For integers $m_1,m_2,n$ with $n\ge1$ and $0\le m_1,m_2\le \binom{n}{2}$, let $\mathcal{G}_{n,m_1,m_2}$ denote the set of directed graphs  on $n$ labelled vertices with $m_1$ `forward' edges (that is, edges $e$ with $s(e) < r(e)$, corresponding to above-diagonal entries of the adjacency matrix) and $m_2$ `backward' ones (edges $e$ with $s(e) > r(e)$, or below-diagonal entries of the adjacency matrix). Loops ($s(e)=r(e)$) are allowed, but not multiple edges. We let $\mathbb{U}_{n,m_1,m_2}$ denote an element of $\mathcal{G}_{n,m_1,m_2}$ chosen randomly according to the uniform distribution, that is,
\begin{equation} \label{eqn:uniform}
\pp(\mathbb{U}_{n,m_1,m_2}=E) = \frac{1}{2^n\binom{n}{m_1}\binom{n}{m_2}}
\end{equation}
for every $E\in\mathcal{G}_{n,m_1,m_2}$ (the factor of $2^n$ accounting for loops).
\end{definition}

It is sometimes useful to think of a directed graph $E$ on vertices $\{1,\dots,n\}$ as being obtained by overlaying three undirected ones ($E_0$ described by the diagonal entries of the adjacency matrix, $E_1$ the above-diagonal entries the and $E_2$ the below-diagonal ones). For example,  if $E_1$ and $E_2$ are both connected, then $E$ is strongly connected. We take up this point of view in the proof of the following.

\begin{theorem} \label{thm:uniform}
If $m_1=m_1(n)$ and $m_2=m_2(n)$ satisfy $0\le m_1,m_2\le \binom{n}{2}$ and
\begin{equation} \label{eqn:mbalance}
\lim_{n\to\infty}\frac{m_i}{\binom{n}{2}} =: \widehat{q}_i \in (0,1),\quad i=1,2,
\end{equation}
then $\cs(\mathbb{U}_{n,m_1,m_2})$ is asymptotically almost surely stably isomorphic to a Cuntz polygon. 
\end{theorem}

\begin{proof}
We adapt the proof of \cite[Lemma 1.2]{Frieze:2016tw}, which addresses undirected graphs. Let $\mathscr{P}=\mathscr{P}_{n,m_1,m_2}\subseteq\mathcal{G}_{n,m_1,m_2}$ be a graph property. Write $N=\binom{n}{2}=\frac{n(n-1)}{2}$ and $q_i=\frac{m_i}{N} \approx \frac{2m_i}{n^2}$, $i=1,2$. Let $A(n)$ be the $n\times n$ matrix whose entries are independent Bernoulli random variables with success probability either $q_1$ (above the diagonal) or $q_2$ (below the diagonal) or $\frac{1}{2}$ (on the diagonal), and let us write $\hh_n$ for the corresponding random graph. We adopt the notational convention suggested after Definition~\ref{def:dnm} and write $\hh_{n,i}$, $i=0,1,2$, for the graphs corresponding to entries on, above or below the diagonal of $A(n)$.

For integers $k_i\ge0$, $i=0,1,2$, with $k_0\le n$ and $k_1,k_2\le N$, the conditional probability that $\hh_n$ has the graph property $\mathscr{P}$, given that  $\hh_{n,i}$ has $k_i$ edges for $i=0,1,2$, is
\begin{align*}
\pp(\hh_n\in\mathscr{P} \mid |\hh_{n,i}^1| &= k_i,\,i=0,1,2)\\
&= \frac{\pp(\hh_n\in\mathscr{P} \wedge |\hh_{n,i}^1| = k_i,\,i=0,1,2)}{\prod_{i=0}^2\pp(|\hh_{n,i}^1| = k_i)}\\
&= \sum_{\substack{E\in\mathscr{P} \\ |E_i^1| = k_i\: \forall i}} \frac{(\frac{1}{2})^{k_0}(\frac{1}{2})^{n-k_0}}{\binom{n}{k_0}(\frac{1}{2})^{k_0}(\frac{1}{2})^{n-k_0}} \cdot \frac{q_1^{k_1}(1-q_1)^{N-k_1}}{\binom{N}{k_1}q_1^{k_1}(1-q_1)^{N-k_1}} \cdot \frac{q_2^{k_2}(1-q_2)^{N-k_2}}{\binom{N}{k_2}q_2^{k_2}(1-q_2)^{N-k_2}}\\
&= \sum_{\substack{E\in\mathscr{P} \\ |E_i^1| = k_i\: \forall i}} \frac{1}{\binom{n}{k_0}\binom{N}{k_1}\binom{N}{k_2}}\\
&= \frac{2^n}{\binom{n}{k_0}} \cdot \pp(\mathbb{U}_{n,k_1,k_2} \in \mathscr{P}) \quad \text{(by definition \eqref{eqn:uniform})}.
\end{align*}
Therefore, by the law of total probability,
\begin{align*}
\pp(\hh_n\in\mathscr{P}) &= \sum_{k_0=0}^n\sum_{k_1,k_2=0}^N \pp(\hh_n\in\mathscr{P} \mid |\hh_{n,i}^1| = k_i,\,i=0,1,2) \cdot \pp(|\hh_{n,i}^1| = k_i,\,i=0,1,2)\\
&= \sum_{k_0=0}^n\sum_{k_1,k_2=0}^N \frac{2^n}{\binom{n}{k_0}} \cdot \pp(\mathbb{U}_{n,k_1,k_2} \in \mathscr{P}) \cdot \pp(|\hh_{n,i}^1| = k_i,\,i=0,1,2)\\
&= \sum_{k_0=0}^n\sum_{k_1,k_2=0}^N \pp(\mathbb{U}_{n,k_1,k_2} \in \mathscr{P}) \cdot \pp(|\hh_{n,1}^1| = k_1) \cdot \pp(|\hh_{n,2}^1| = k_2)\\
&\ge \pp(\mathbb{U}_{n,m_1,m _2} \in \mathscr{P}) \cdot \pp(|\hh_{n,1}^1| = m_1) \cdot \pp(|\hh_{n,2}^1| = m_2)\\
&\ge 	\frac{1}{100\sqrt{m_1m_2}} \pp(\mathbb{U}_{n,m_1,m _2} \in \mathscr{P}) \quad \text{for sufficiently large $n$},
\end{align*}
the latter estimate obtained from Stirling's formula applied to the binomial probabilities
\[
\pp(|\hh_{n,i}^1| = m_i) = \binom{N}{m_i}q_i^{m_i}(1-q_i)^{N-m_i} = \frac{N!}{m_i!(N-m_i)!} q_i^{m_i}(1-q_i)^{N-m_i}, \,i=1,2,
\]
as in the proof of \cite[Lemma 1.2]{Frieze:2016tw}. Note that \eqref{eqn:mbalance} justifies this use of Stirling's formula, as it ensures that the numbers $m_1$, $m_2$, $N-m_1$ and $N-m_2$ are all large when $n$ is. What we have now shown is that
\begin{equation} \label{eqn:contiguity}
\pp(\mathbb{U}_{n,m_1,m _2} \in \mathscr{P}) \le 100\sqrt{m_1m_2} \cdot \pp(\hh_{n,m_1,m_2}\in\mathscr{P}) \quad\text{for all large $n$}.
\end{equation} 
By assumption \eqref{eqn:mbalance}, $M(n)=A(n)^t-I$ (for that matter, $A(n)$ itself) satisfies the requirement 	\eqref{eqn:balance-bernoulli} of Theorem~\ref{thm:bernoulli} for a suitable $\varepsilon < \min\left\{\widehat{q}_1,\widehat{q}_2,1-\widehat{q}_1,1-\widehat{q}_2\right\}$. As mentioned in the proof of Theorem~\ref{thm:bernoulli}, the asymptotic probability that $M(n)$ is singular (which is the relevant graph property $\mathscr{P}$) is $(\sqrt{1-\varepsilon}+o(1))^n$. By \eqref{eqn:contiguity}, the corresponding asymptotic probability for $\mathbb{U}_{n,m_1,m _2}$ is at most $100\sqrt{m_1m_2}(\sqrt{1-\varepsilon}+o(1))^n$, which converges to $0$ as $n\to\infty$. We also know that, with high probability, the undirected graphs $(\mathbb{U}_{n,m_1,m _2})_1$ and $(\mathbb{U}_{n,m_1,m _2})_2$ are connected (see \cite[Theorem 4.1]{Frieze:2016tw}), so $\mathbb{U}_{n,m_1,m _2}$ is strongly connected (and does not consist of a single cycle). The theorem now follows from Proposition~\ref{prop:kpoly}.\ref{poly3}.
\end{proof}

\section{Empirical \texorpdfstring{$K$}{K}-data} \label{section:stats}

In this section, we present some of the data produced by our computer code. We generated samples of size $m=10^5$ for various random Bernoulli digraphs $\dd_{n,q}$, Erd\H{o}s--R\'{e}nyi graphs $\ee_{n,q}$ and regular multigraphs $\g_{n,r}$, and collected the $K$-theoretic data outlined in Section~\ref{section:intro} for several small primes (up to $p=37$) and a couple larger ones ($p=137$ and $p=277$). We also collected the supplemental data outlined in Sections~\ref{subsection:unit},~\ref{subsection:flow},~\ref{subsection:exact} needed to tally flow equivalence to a full shift or exact isomorphism to a Cuntz algebra or Cuntz polygon. Typical sampled graphs are illustrated in Figure~\ref{fig:graph} and Figure~\ref{fig:graph2}.

All collected data were consistent with available theory, as long as the number of vertices was sufficiently large.  In practice, $n=50$ vertices seemed to be large enough to produce expected results (see Table~\ref{table:13ci}), though most of the time we opted for $n=100$. Note that, since our sample sizes are very large, the margin of error for the data is small. Using the normal approximation to the binomial distribution, at the $99\%$ level of confidence this margin is at most
\[
z_{0.005}\left.\sqrt{\frac{x(1-x)}{n}}\right\vert_{x=1/2} \approx \frac{2.576}{2\sqrt{10^5}} \approx 0.004
\]
for any of our proportions of interest. This means that if our sampling procedure were to be repeated many times, and $\widehat{\gamma}$ denotes our empirical estimate of some true probability $\gamma$, then the interval $[\widehat{\gamma}-0.004, \widehat{\gamma}+0.004]$ should contain $\gamma$ approximately 99\% of the time. In practice, the margin of error depends on the variability of the data characteristic being measured, possibly resulting in a narrower confidence interval. See, for example, Table~\ref{table:13ci}, and notice also that the error bars in graphs such as Figure~\ref{fig:bargrapher2} vary from bar to bar.

\subsection{Bernoulli graphs $\dd_{n,q}$}

Here, our emphasis was on  \eqref{eqn:ercuntz} for small primes $p$ and \eqref{eqn:ersylow} for small $p$-groups $G=\zz/(p^N\zz)$ and $G=(\zz/p\zz)^N$. See Figure~\ref{fig:bargrapher2} and Table~\ref{table:dq}, respectively. Conjecture~\ref{conjecture:fullshift} and  Conjecture~\ref{conjecture:bernoulli} are supported by the data in Table~\ref{table:dnqextra} (see also Table~\ref{table:dnqextra2}).

As discussed in Remark~\ref{remark:logn}, sparser Bernoulli digraphs (those for which $q=\log n/n$) exhibit rather different statistics. This expectation was consistent with the data, some of which is presented in Table~\ref{table:dlog}. As for connectivity, note that if $q=(\log n + \omega)/n$ for some function $\omega=\omega(n)$, then
\[
\lim_{n\to\infty} \pp\left(\dd_{n,q} \text{ is strongly connected}\right) = 
 \begin{cases}
 0 & \text{ if }\: \omega\to-\infty\\
 e^{-2e^{-c}} & \text{ if }\: \omega\to c \text{ constant}\\
 1 & \text{ if }\: \omega\to\infty
 \end{cases}
\]
(see \cite[Theorem 13.9]{Frieze:2016tw}, which is derived from the similar behaviour \cite[Theorem 4.1]{Frieze:2016tw} exhibited by the symmetric versions $\ee_{n,q}$). If $n$ is large enough and $q=\log n/n$ (that is, $\omega=0$), we would expect strong connectivity about $e^{-2}\approx13.5\%$ of the time. Our observation of $15.6\%$ for $n=100$ (see Table~\ref{table:dlog}) is not inconsistent with this.

\subsection{Erd\H{o}s--R\'{e}nyi graphs $\ee_{n,q}$}

From \eqref{eqn:ncyclic} and \eqref{eqn:allcyclic}, we can compute the asymptotic cyclicity probabilities for $K_0(\cs(\ee_{n,q}))_p$, for any prime $p$. The recorded data closely agree with these probabilities (see Figure~\ref{fig:bargraphsymer2}). As mentioned in Section~\ref{subsection:exact}, it is expected, though not yet proved, that the asymptotic probability that $K_0(\cs(\ee_{n,q}))$ is cyclic is equal to the product of the respective probabilities over all primes $p$, that is, to
\[
\prod_{p\text{ prime}} \prod_{k=2}^\infty\left(1-p^{-2k+1}\right) \approx 0.79352.
\]
The data in Table~\ref{table:eq} and Table~\ref{table:enqextra} are in accordance with this expectation. Table~\ref{table:enqextra} also contains data for the open Question~\ref{q1} and data in support of Conjecture~\ref{conjecture:bernoulli}.

\subsection{Regular multigraphs $\g_{n,r}$}

The graph algebras $\cs(\g_{n,r})$ are almost never exactly isomorphic to Cuntz polygons (see Question~\ref{q2} and Table~\ref{table:gnrextra}). The asymptotic behaviour of the determinant is also somewhat opaque (see Table~\ref{table:gnrextra} again), so we refrain from conjecturing the asymptotic probability of flow equivalence to a full shift for these graphs, leaving open Question~\ref{q1}.

On the other hand, the purely $K$-theoretic data are easier to understand. With $n=100$, the data collected for $K_0(\cs(\g_{n,r}))_{p}$ for all primes $p$ between $3$ and $37$ (and also $p=137$ and $p=277$) were remarkably in line with the limiting distributions \eqref{eqn:ncyclic} and \eqref{eqn:cyclicn} provided that $p\nmid r-1$. See, for example, Table~\ref{table:13ci}, Figure~\ref{fig:bargraph5} and Figure~\ref{fig:bargraph11}.

For $p=2$, the event ``$K_0(\cs(\g_{n,r}))_{2}$ is cyclic'' tended to be concentrated on a single outcome $K_0(\cs(\g_{n,r}))_{2}=\zz/2^N\zz$, with $N=0$ if $2\nmid r-1$. (On the other hand, note that, in general, $K_0(\cs(\g_{n,r}))_{p}$ cannot be trivial if $p\mid r-1$.) We did however observe some interesting deviations from this pattern (see Figure~\ref{fig:bargraph2a} and Figure~\ref{fig:bargraph2b}), including variable behaviour for a fixed value of $r$ and different $n$ (compare Figure~\ref{fig:bargraph2c} with Figure~\ref{fig:bargraph2d}).

In Table~\ref{table:adams}, $\widehat\gamma_{n,r}$ and $\widehat\pi_{p,r}$ denote our estimators of the probabilities
\[
\gamma_{n,r} := \pp\left(K_0(\cs(\g_{n,r})) \text{ is cyclic}\right)
\]
and
\[
\pi_{p,r} :=  \lim_{n\in2\nn}\pp\left(K_0(\cs(\g_{n,r}))_p \text{ is cyclic}\right).
\]
While we are somewhat begging the question here, as the definition of $\pi_{p,r}$ does assume that the limit exists, note that, from Proposition~\ref{prop:pcyclic},
\begin{equation} \label{eqn:pipr}
\pi_{p,r}=\prod_{k=2}^\infty\left(1-p^{-2k+1}\right)
\end{equation}
if $p\nmid 2(r-1)$. Convergence outside of this setting, that is, for $p=2$ or $p\mid r-1$, does seem to be supported by the data. In fact, if we rely on the heuristic principle that $n=100$ is large enough to indicate asymptotic behaviour, then from Table~\ref{table:adams} we see that $\pi_{2,r}$ is consistently about $0.42$, independently of $r$. As for the provenance of this number, we suspect it to be \eqref{eqn:pipr} adjusted to also include the $k=1$ term. Fascinatingly, this adjustment also appears to govern the asymptotic probability of cyclicity of $K_0(\cs(\g_{n,r}))_{p}$ for primes $p$ dividing $r-1$ (see Table~\ref{table:newdist}). Arguing heuristically as in the proof of Theorem~\ref{thm:bernoulli}.\ref{ber5}, we would also expect $\gamma_{n,r}$ to converge as $n\to\infty$ to $\prod_{p \text{ prime}} \pi_{p,r} =:\gamma_r$. This expectation is indeed supported by the data (see Table~\ref{table:adams} and Table~\ref{table:2k1}). Putting this together, we are led to make the following.

\begin{conjecture} \label{conjecture:regular}
For an integer $r$ and prime $p$, let $\pi_{p,r} := \lim_{n\in2\nn}\pp\left((K_0(\cs(\g_{n,r}))_p \text{ is cyclic}\right)$. Then, for any $r\ge3$:
\begin{enumerate}[1.]
\item $\pi_{p,r} = \prod_{k=1}^\infty\left(1-p^{-2k+1}\right)$ if $p=2$ (in which case, $\pi_{p,r}\approx 0.419$) or if $p$ is odd and $p\mid r-1$;
\item the asymptotic probability that $\cs(\g_{n,r})$ is stably isomorphic to a Cuntz algebra is
\begin{equation} \label{eqn:gnrcuntz}
\gamma_r := \prod_{p \text{ prime}} \pi_{p,r} = \prod_{\substack{p \text{ prime} \\ p \mid 2(r-1)}}\left(1-p^{-1}\right) \prod_{p \text{ prime}} \prod_{k=2}^\infty\left(1-p^{-2k+1}\right) \:(\approx 0.397 \text{ if $r=2^j+1$}).
\end{equation}
\end{enumerate}
\end{conjecture}

\begin{figure}[!htbp]
\centering
\includegraphics[scale=0.8]{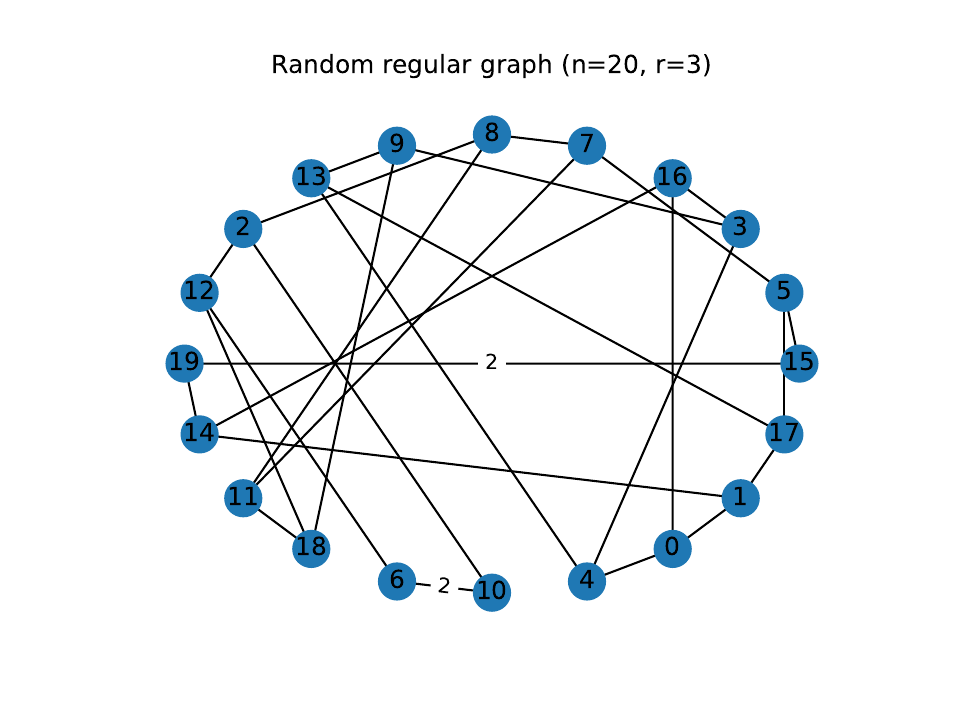}
\caption{Sample generated graph $\g_{20,3}$} \label{fig:graph}
\end{figure}

\begin{figure}[!htbp]
\centering
\includegraphics[scale=0.5]{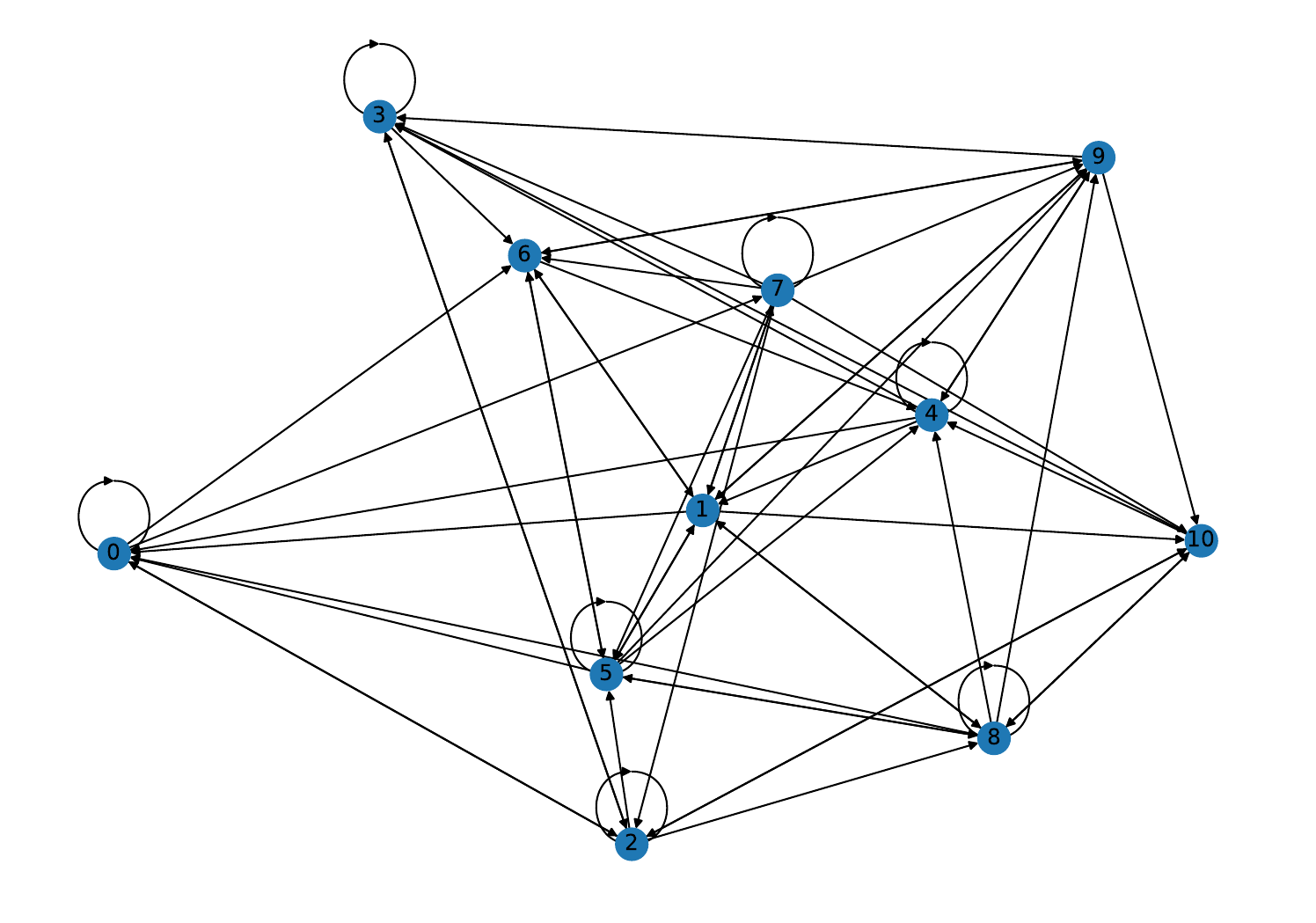}
\caption{Sample generated graph $\dd_{11,1/2}$} \label{fig:graph2} 
\end{figure}

\begin{figure}[!htbp]
\centering
\includegraphics[scale=0.8]{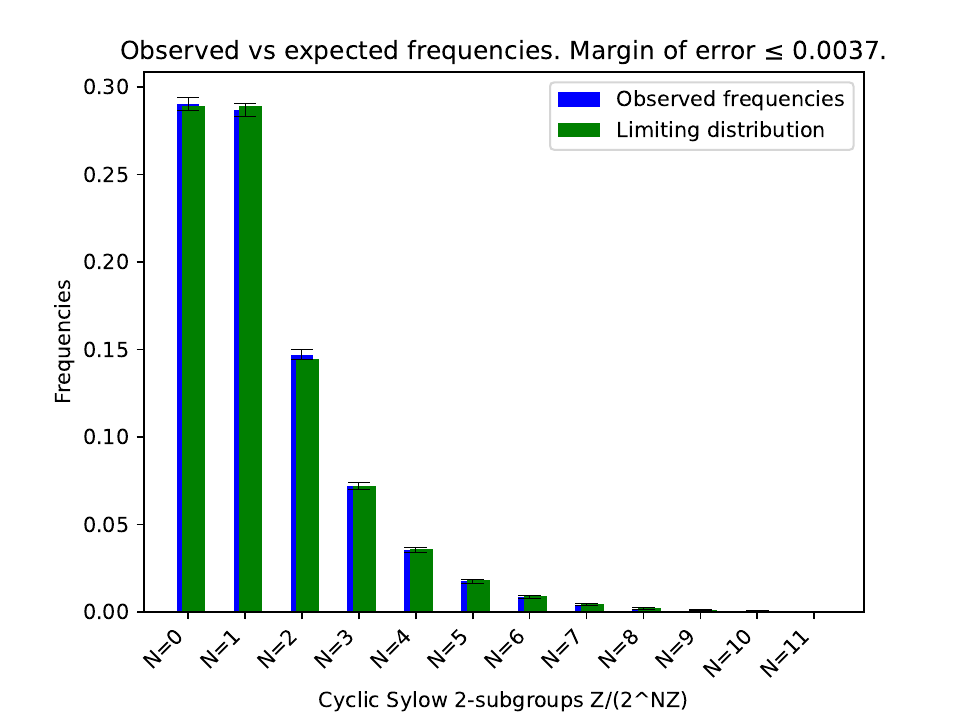}
\caption{Frequency distribution for $10^5$ observations of $K_0(\cs(\dd_{100,1/4}))_{2}$} \label{fig:bargrapher2}
\end{figure}

\begin{table}[!htbp]
\begin{center}
\begin{tabular}{c|c|c|ccccc}
\toprule
\multirow{2}{*}{$(n,q)$} &
\multirow{2}{*}{$\dd_{n,q}$ connected} &
\multirow{2}{*}{$K_1\ne0$} &
\multicolumn{5}{|c}{$(\tor(K_0))_p$ cyclic} \\ \cline{4-8}
& & & all $p$ & $p=2$ & $p=3$ & $p=5$ & $p=7$\\
\hline
$(50,1/2)$ & 100000 & 0 & 84881 & 86769 & 98104 & 99788 & 99954\\
$(100,1/2)$ & 100000 & 0 & 85098 & 86928 & 98086 & 99819 & 99961\\
$(50,1/3)$ & 100000 & 0 & 84597 & 86598 & 97975 & 99784 & 99950\\
$(100,1/3)$ & 100000 & 0 & 84727 & 86676 & 98003 & 99801 & 99952\\
$(50,1/4)$ & 99994 & 0 & 84756 & 86679 & 98057 & 99793 & 99955\\
$(100,1/4)$ & 100000 & 0 & 84586 & 86570 & 97982 & 99805 & 99958\\
\hline
$n\to\infty$ & $10^5-O\left(\frac{1}{n}\right)$ & $10^5\left(\frac{1}{\sqrt{2}}+o(1)\right)^n$ & $84694$ & $86636$ & 98022 & 99794 & 99951\\
\bottomrule
\end{tabular}
\caption{Selected totals for $10^5$ observations of $\cs(\dd_{n,1/k})$, $n=50,100$, $k=2,3,4$} \label{table:dq}
\end{center}
\end{table}

\begin{table}[!htbp]
\begin{center}
\begin{tabular}{c|c|c|c|c|c|c|c}
\toprule
$(n,q)$ & Connected & $K_1\ne0$ & $K_0$ cyclic & $\mathrm{Det}<0$ & Full shift & $\ppm$ & $\oo_m$ \\
\hline
$(50,1/2)$ & 100000 & 0 & 84377 & \textcolor{teal}{49887} & \textcolor{teal}{42186} & \textcolor{purple}{50880} & \textcolor{purple}{43518} \\
$(100,1/2)$ & 100000 & 0 & 84525 & \textcolor{teal}{49963} & \textcolor{teal}{42190} & \textcolor{purple}{50640} & \textcolor{purple}{43481} \\
$(50,1/3)$ & 100000 & 0 & 84825 & \textcolor{teal}{50036} & \textcolor{teal}{42381} & \textcolor{purple}{50695} & \textcolor{purple}{43503} \\
$(100,1/3)$ & 100000 & 0 & 84736 & \textcolor{teal}{50044} & \textcolor{teal}{42410} & \textcolor{purple}{50736} & \textcolor{purple}{43635} \\
$(50,1/4)$ & 99989 & 1 & 84610 & \textcolor{teal}{49966} & \textcolor{teal}{42225} & \textcolor{purple}{50739} & \textcolor{purple}{43581} \\
$(100,1/4)$ & 100000 & 0 & 84636 & \textcolor{teal}{49831} & \textcolor{teal}{42260} & \textcolor{purple}{50764} & \textcolor{purple}{43582} \\
\hline
$n\to\infty$ & $10^5-O\left(\frac{1}{n}\right)$ & $10^5\left(\frac{1}{\sqrt{2}}+o(1)\right)^n$ & $84694$ & \textcolor{teal}{50000} & \textcolor{teal}{42347} & \textcolor{purple}{51451} & \textcolor{purple}{43576} \\
\bottomrule
\end{tabular}
\caption{Selected totals for $10^5$ observations of $\cs(\dd_{n,1/k})$, $n=50,100$, $k=2,3,4$. Numbers marked in \textcolor{purple}{purple} support Conjecture~\ref{conjecture:bernoulli} and those marked in \textcolor{teal}{teal} support Conjecture~\ref{conjecture:fullshift}.} \label{table:dnqextra}
\end{center}
\end{table}

\begin{table}[!htbp]
\begin{center}
\begin{tabular}{c|c|c|c|c|c|c|c}
\toprule
$(n,q)$ & Connected & $K_1\ne0$ & $K_0$ cyclic & $\mathrm{Det}<0$ & Full shift & $\ppm$ & $\oo_m$ \\
\hline
$(50,1/2)$ & 100000 & 0 & 84794 & \textcolor{teal}{50112} & \textcolor{teal}{42478} & \textcolor{purple}{50588} & \textcolor{purple}{43616} \\
$(100,1/2)$ & 100000 & 0 & 84524 & \textcolor{teal}{49976} & \textcolor{teal}{42302} & \textcolor{purple}{50429} & \textcolor{purple}{43337} \\
$(50,1/3)$ & 100000 & 0 & 84726 & \textcolor{teal}{49632} & \textcolor{teal}{42065} & \textcolor{purple}{50709} & \textcolor{purple}{43535} \\
$(100,1/3)$ & 100000 & 0 & 84669 & \textcolor{teal}{49989} & \textcolor{teal}{42394} & \textcolor{purple}{50498} & \textcolor{purple}{43421} \\
$(50,1/4)$ & 99992 & 5 & 84631 & \textcolor{teal}{49833} & \textcolor{teal}{42205} & \textcolor{purple}{50742} & \textcolor{purple}{43553} \\
$(100,1/4)$ & 100000 & 0 & 84706 & \textcolor{teal}{50051} & \textcolor{teal}{42377} & \textcolor{purple}{50437} & \textcolor{purple}{43309} \\
\hline
$n\to\infty$ & $10^5-O\left(\frac{1}{n}\right)$ & $10^5\left(\frac{1}{\sqrt{2}}+o(1)\right)^n$ & $84694$ & \textcolor{teal}{50000} & \textcolor{teal}{42347} & \textcolor{purple}{51451} & \textcolor{purple}{43576}\\
\bottomrule
\end{tabular}
\caption{Selected totals for $10^5$ observations of shifted Bernoulli digraph algebras $\cs(\dd_{n,1/k}+I)$, $n=50,100$, $k=2,3,4$. Numbers marked in \textcolor{purple}{purple} are relevant to Conjecture~\ref{conjecture:bernoulli} and  those marked in \textcolor{teal}{teal} support Theorem~\ref{thm:bf}.} \label{table:dnqextra2}
\end{center}
\end{table}

\begin{table}[!htbp]
\begin{center}
\begin{tabular}{c|c|c|ccccc}
\toprule
\multirow{2}{*}{$q$ ($n=100$)} &
\multirow{2}{*}{$\dd_{100,q}$ connected} &
\multirow{2}{*}{$K_1\ne0$} &
\multicolumn{5}{|c}{$(\tor(\coker(A^t-I)))_p$ cyclic$^*$} \\ \cline{4-8}
& & & all $p$ & $p=2$ & $p=3$ & $p=5$ & $p=7$\\
\hline
$3\log n/n$ & 99993 & 3 & 84617 & 86586 & 97990 & 99781 & 99958\\
$2\log n/n$ & 98606 & 114 & 84880 & 86786 & 98062 & 99805 & 99952\\
$\log n/n$ & 15623 & 8829 & 85713 & 87481 & 98183 & 99801 & 99941\\
$\log n/2n$ & 0 & 51702 & 87968 & 89172 & 98776 & 99866 & 99978\\
\bottomrule
\end{tabular}
\caption{$10^5$ observations of $\cs(\dd_{n,k\log n/n})$, $n=100$, $k=3,2,1,0.5$.\vspace{0.1cm}\ \ \hspace{\textwidth} \footnotesize{$^*$We may no longer have $K_0 \cong \coker(A^t-I)$ as sinks now occur with nonzero probability so would have to be taken into account as in \cite[Theorem 3.2]{Raeburn:2004tg}.}} \label{table:dlog}
\end{center}
\end{table}

\begin{table}[!htbp]
\begin{center}
\begin{tabular}{c|c|c|ccccc}
\toprule
\multirow{2}{*}{$(n,q)$} &
\multirow{2}{*}{$\ee_{n,q}$ connected} &
\multirow{2}{*}{$K_1\ne0$} &
\multicolumn{5}{|c}{$(\tor(K_0))_p$ cyclic} \\ \cline{4-8}
& & & all $p$ & $p=2$ & $p=3$ & $p=5$ & $p=7$\\
\hline
$(50,1/2)$ & 100000 & 0 & 79239 & 83805 & 95846 & 99169 & 99696\\
$(100,1/2)$ & 100000 & 0 & 79325 & 83786 & 95889 & 99125 & 99707\\
$(50,1/3)$ & 100000 & 0 & 79235 & 83784 & 95854 & 99157 & 99703\\
$(100,1/3)$ & 100000 & 0 & 79206 & 83797 & 95883 & 99139 & 99694\\
$(50,1/4)$ & 99997 & 3 & 79306 & 83833 & 95898 & 99150 & 99687\\
$(100,1/4)$ & 100000 & 0 & 79331 & 83822 & 95839 & 99169 & 99697\\
\hline
$n\to\infty$ & $10^5-O\left(\frac{1}{n}\right)$ & $o(1)$ & 79352 & 83884 & 95851 & 99167 & 99702\\
\bottomrule
\end{tabular}
\caption{Selected totals for $10^5$ observations of $\cs(\ee_{n,1/k})$, $n=50,100$, $k=2,3,4$} \label{table:eq}
\end{center}
\end{table}

\begin{table}[!htbp]
\begin{center}
\begin{tabular}{c|c|c|c|c|c|c|c}
\toprule
$(n,q)$ & Connected & $K_1\ne0$ & $K_0$ cyclic & $\mathrm{Det}<0$ & Full shift & $\ppm$ & $\oo_m$ \\
\hline
$(50,1/2)$ & 100000 & 0 & 79379 & \textcolor{teal}{56708} & \textcolor{teal}{45005} & \textcolor{purple}{59431} & \textcolor{purple}{48330} \\
$(75,1/2)$ & 100000 & 0 & 79325 & \textcolor{teal}{50912} & \textcolor{teal}{40336} & \textcolor{purple}{59051} & \textcolor{purple}{48073} \\
$(100,1/2)$ & 100000 & 0 & 79416 & \textcolor{teal}{45762} & \textcolor{teal}{36274} & \textcolor{purple}{59317} & \textcolor{purple}{48281} \\
$(50,1/3)$ & 100000 & 0 & 79286 & \textcolor{teal}{55144} & \textcolor{teal}{43670} & \textcolor{purple}{59335} & \textcolor{purple}{48317} \\
$(75,1/3)$ & 100000 & 0 & 79400 & \textcolor{teal}{45625} & \textcolor{teal}{36269} & \textcolor{purple}{59195} & \textcolor{purple}{48253} \\
$(100,1/3)$ & 100000 & 0 & 79371 & \textcolor{teal}{46400} & \textcolor{teal}{36794} & \textcolor{purple}{59121} & \textcolor{purple}{48224} \\
$(50,1/4)$ & 99999 & 1 & 79526 & \textcolor{teal}{46138} & \textcolor{teal}{36771} & \textcolor{purple}{59643} & \textcolor{purple}{48634} \\
$(75,1/4)$ & 100000 & 0 & 79296 & \textcolor{teal}{47678} & \textcolor{teal}{37797} & \textcolor{purple}{59246} & \textcolor{purple}{48315} \\
$(100,1/4)$ & 100000 & 0 & 79322 & \textcolor{teal}{54237} & \textcolor{teal}{42983} & \textcolor{purple}{59275} & \textcolor{purple}{48218} \\
\hline
$n\to\infty$ & $10^5-O\left(\frac{1}{n}\right)$ & $o(1)$ & 79352 & \textcolor{teal}{$10^5\delta_{n,q}$} & \textcolor{teal}{$10^5\sigma_{n,q}$} & \textcolor{purple}{60793} & \textcolor{purple}{48240} \\
\bottomrule
\end{tabular}
\caption{Selected totals for $10^5$ observations of $\cs(\ee_{n,1/k})$, $n=50,100$, $k=2,3,4$. Numbers marked in \textcolor{purple}{purple} support Conjecture~\ref{conjecture:bernoulli}. Numbers marked in \textcolor{teal}{teal} provide empirical data for the open Question~\ref{q1}. In particular, using the definitions \eqref{eqn:enq-shiftprob}, the values $\delta_{n,q}$ and $\sigma_{n,q}$ in the data table obey the relationship $\sigma_{n,q}\approx0.79352\cdot\delta_{n,q}$, supporting statistical independence of determinant negativity and cyclicity (cf.\ \eqref{eqn:ecyclic}). } \label{table:enqextra}
\end{center}
\end{table}

\begin{figure}[!htbp]
\centering
\includegraphics[scale=0.8]{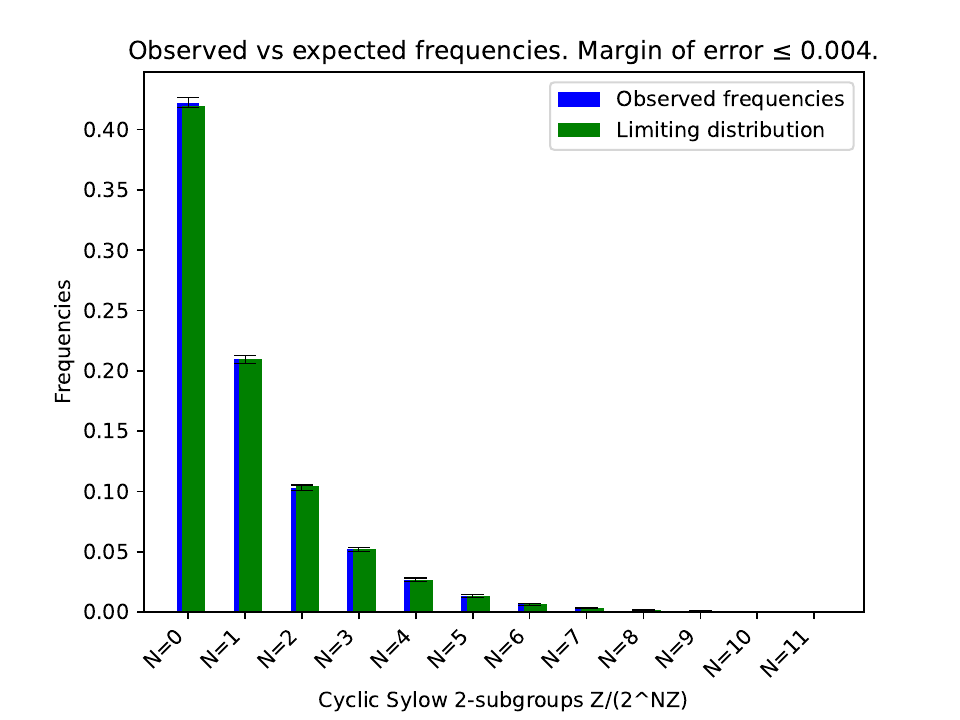}
\caption{Frequency distribution for $10^5$ observations of $K_0(\cs(\ee_{100,1/2}))_{2}$} \label{fig:bargraphsymer2}
\end{figure}

\begin{table}[!htbp]
\begin{center}
\begin{tabular}{c|ccccc}
\toprule
\multirow{2}{*}{$n$} &
\multicolumn{5}{c}{$K_0(\cs(\g_{n,3}))_{13}$} \\ \cline{2-6}
& $0$ & $\zz/13\zz$ & $\zz/13^2\zz$ & $(\zz/13\zz)^2$ & $\zz/13^3\zz$\\
\hline
$n\to\infty$ & 0.92265 & 0.07097 & 0.00546 & 0.00042 & 0.00042\\
200 & $[0.9223, 0.9266]$ & $[0.0672, 0.0713]$ & $[0.0048, 0.0060]$ & $[0.0002, 0.0006]$ & $[0.0002, 0.0006]$\\
100 & $[0.9192, 0.9236]$ & $[0.0700, 0.0742]$ & $[0.0049, 0.0061]$ & $[0.0004, 0.0007]$ & $[0.0002, 0.0005]$\\
60 & $[0.9215, 0.9258]$ & $[0.0677, 0.0718]$ & $[0.0053, 0.0066]$ & $[0.0002, 0.0005]$ & $[0.0002, 0.0005]$\\
50 & $[0.9225, 0.9268]$ & $[0.0672, 0.0713]$ & $[0.0048, 0.0060]$ & $[0.0003, 0.0006]$ & $[0.0001, 0.0004]$\\
20 & $[0.9583, 0.9615]$ & $[0.0360, 0.0391]$ & $[0.0021, 0.0029]$ & $[0.0000, 0.0002]$ & $[0.0000, 0.0000]$\\
10 & $[1.0000, 1.0000]$ & $[0.0000, 0.0000]$ & $[0.0000, 0.0000]$ & $[0.0000, 0.0000]$ & $[0.0000, 0.0000]$\\
\bottomrule
\end{tabular}
\caption{$99\%$ CIs for Sylow 13-subgroups of $K_0(\cs(\g_{n,3}))$, various $n$} \label{table:13ci}
\end{center}
\end{table}

\begin{figure}[!htbp]
\centering
\includegraphics[scale=0.8]{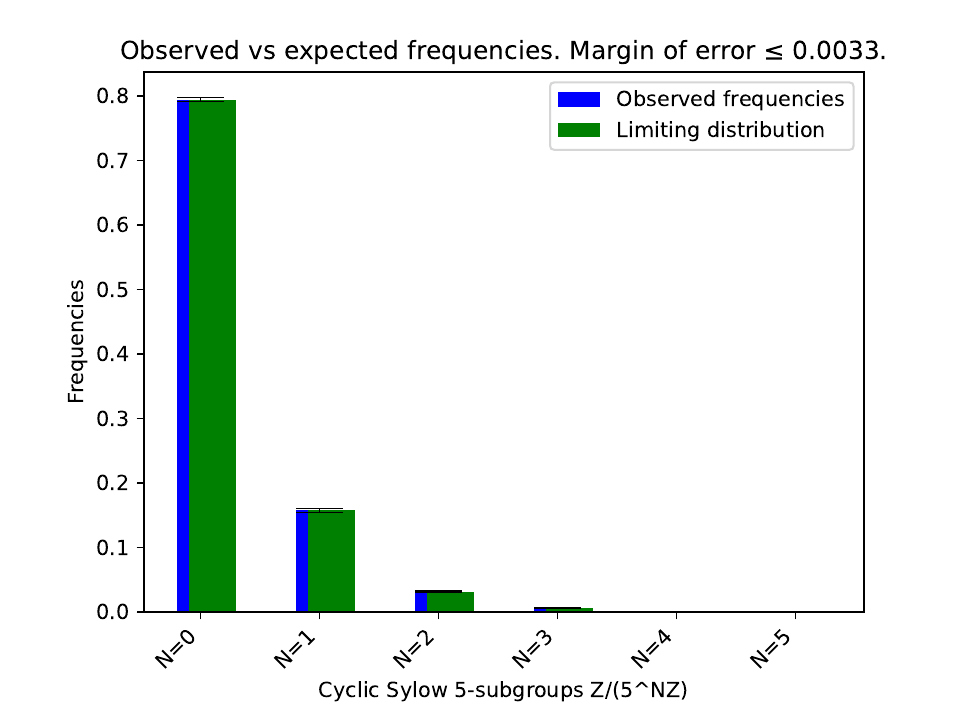}
\caption{Frequency distribution for $10^5$ observations of $K_0(\cs(\g_{200,3}))_{5}$} \label{fig:bargraph5}
\end{figure}

\begin{figure}[!htbp]
\centering
\includegraphics[scale=0.8]{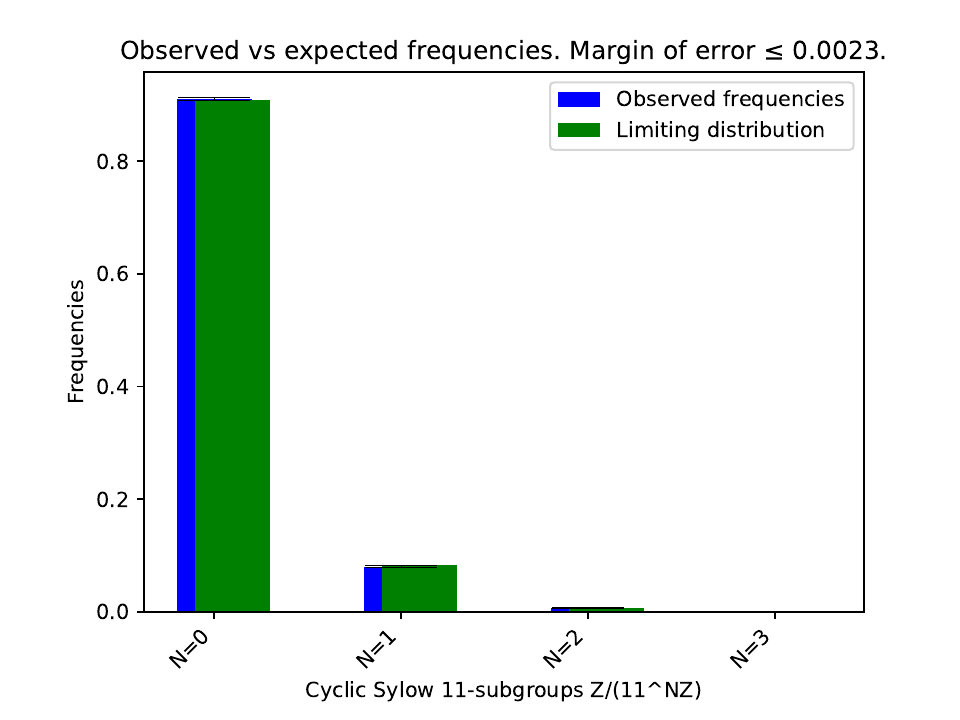}
\caption{Frequency distribution for $10^5$ observations of $K_0(\cs(\g_{100,11}))_{11}$} \label{fig:bargraph11}
\end{figure}

\begin{figure}[!htbp]
\centering
\includegraphics[scale=0.8]{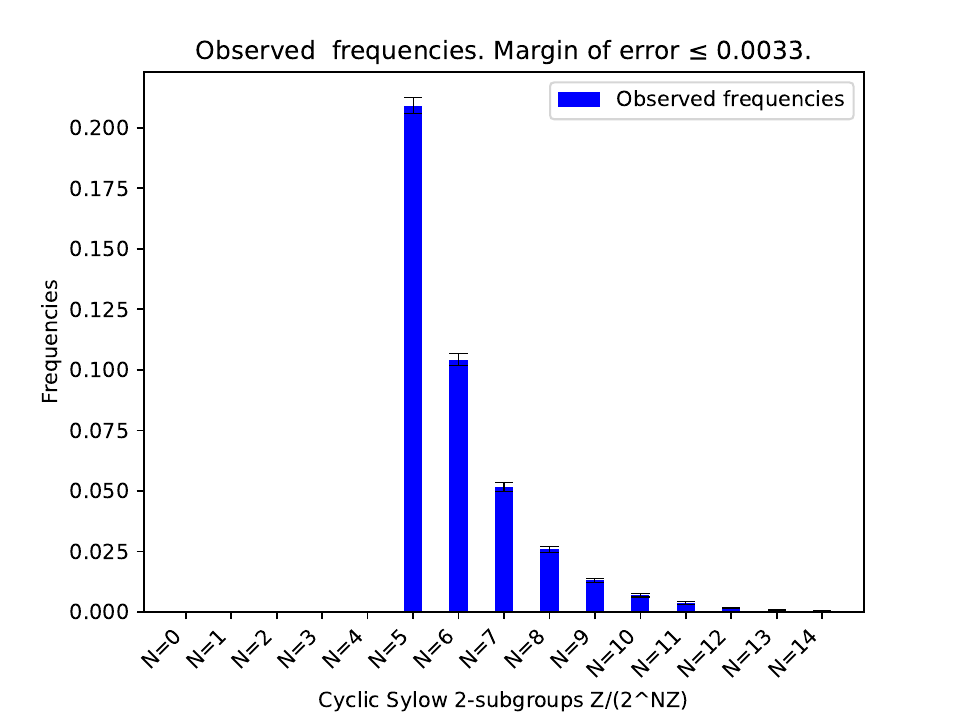}
\caption{Frequency distribution for $10^5$ observations of $K_0(\cs(\g_{100,5}))_{2}$ in support of Conjecture~\ref{conjecture:regular}} \label{fig:bargraph2a}
\end{figure}

\begin{figure}[!htbp]
\centering
\includegraphics[scale=0.8]{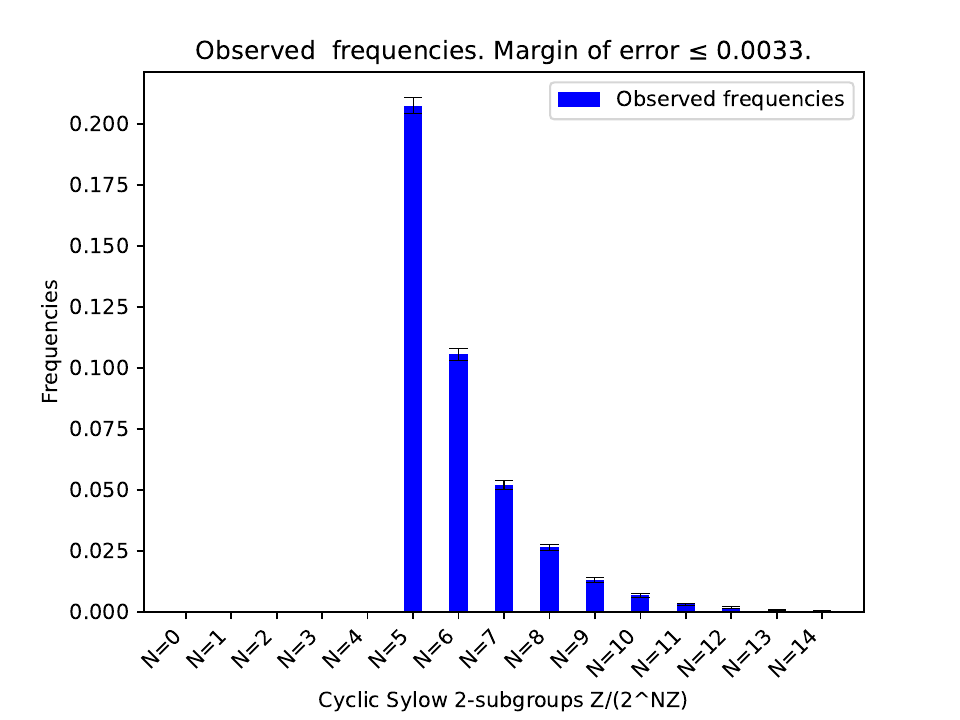}
\caption{Frequency distribution for $10^5$ observations of $K_0(\cs(\g_{100,13}))_{2}$ in support of Conjecture~\ref{conjecture:regular}} \label{fig:bargraph2b}
\end{figure}

\begin{figure}[!htbp]
\centering
\includegraphics[scale=0.8]{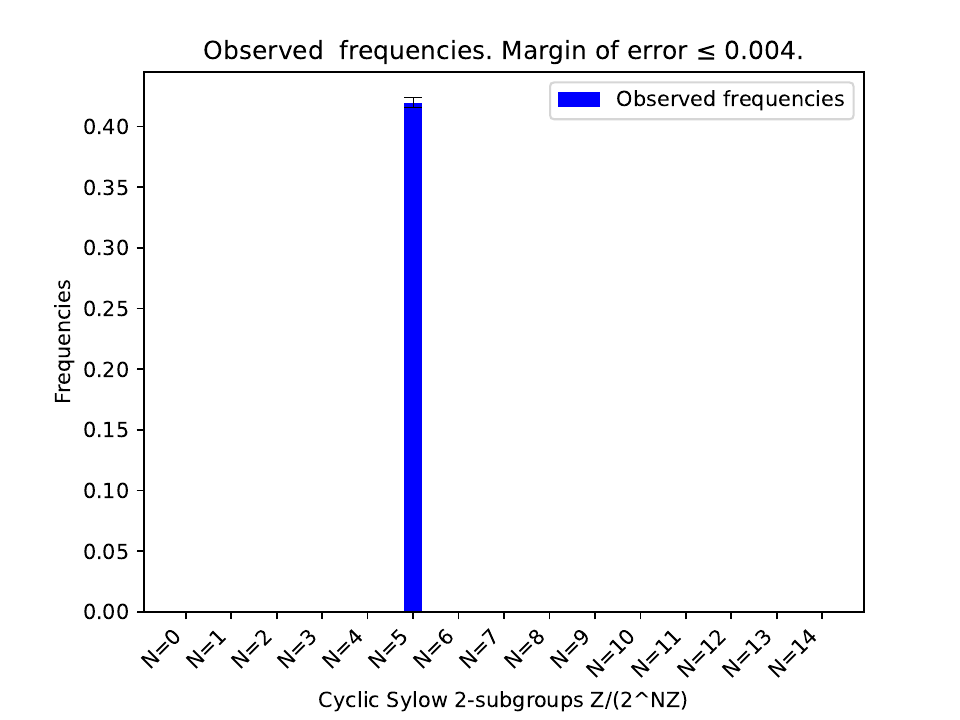}
\caption{Frequency distribution for $10^5$ observations of $K_0(\cs(\g_{100,9}))_{2}$ in support of Conjecture~\ref{conjecture:regular}} \label{fig:bargraph2c}
\end{figure}

\begin{figure}[!htbp]
\centering
\includegraphics[scale=0.8]{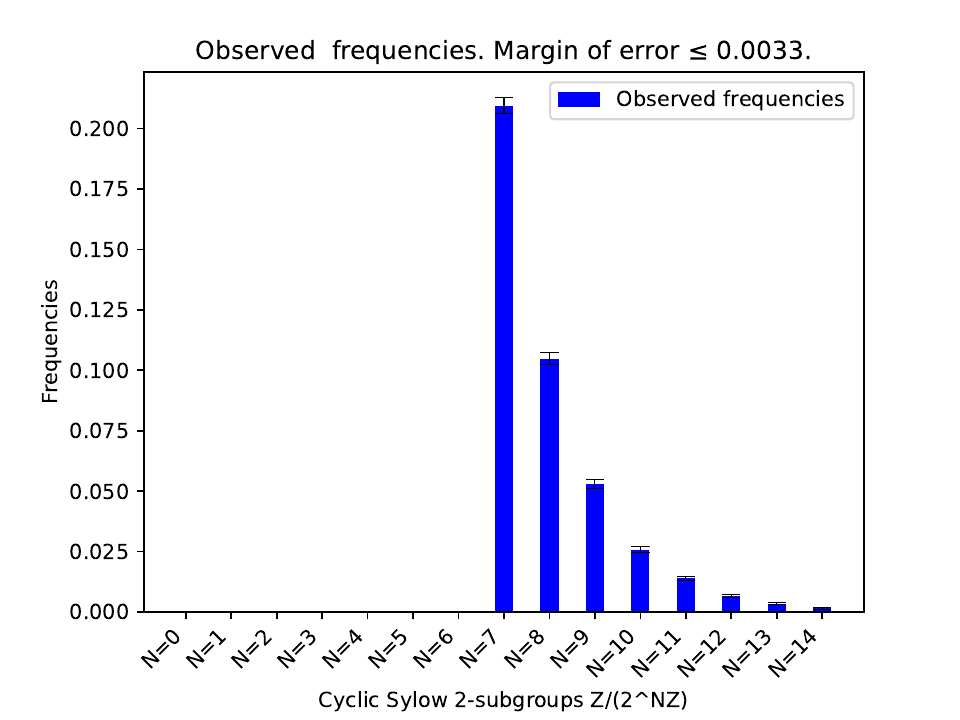}
\caption{Frequency distribution for $10^5$ observations of $K_0(\cs(\g_{200,9}))_{2}$} \label{fig:bargraph2d}
\end{figure}

\begin{table}[!htbp]
\begin{center}
\begin{tabular}{c|ccccccccccc}
\toprule
$r$ & 4 & 5 & 6 & 7 & 8 & 9 & 10 & 11 & 12 & 17\\
\hline
$\widehat\pi_{2,r}$ & \bf 0.416 & \bf 0.418 & \bf 0.419 & \bf 0.415 & \bf 0.416 & \bf 0.420 & \bf 0.420 & \bf 0.420 & \bf 0.418 & \bf 0.419\\
$\prod_{\stackrel{p\text{ prime}}{p\le37}}\widehat\pi_{p,r}$ & 0.264 & 0.395 & 0.317 & 0.262 & 0.338 & 0.396 & 0.265 & 0.318 & 0.359 & 0.397\\
$\widehat\gamma_{100,r}$ & 0.265 & 0.395 & 0.316 & 0.261 & 0.338 & 0.396 & 0.264 & 0.317 & 0.360 & 0.397\\
\bottomrule
\end{tabular}
\caption{Cyclicity frequencies for $10^5$ observations of $K_0(\cs(\g_{100,r}))$, various $r$. Marked in bold are empirical estimates supporting Conjecture~\ref{conjecture:regular}.} \label{table:adams}
\end{center}
\end{table}

\begin{table}[!htbp]
\begin{center}
\begin{tabular}{c|cccccccccccc}
\toprule
$r$ & 6 & 7 & 8 & 10 & 11 & 12 & 13 & 14 & 20\\
$p \mid r-1$ & 5 & 3 & 7 & 3 & 5 & 11 & 3 & 13 & 19\\
$\widehat\pi_{p,r}$ & \bf 0.794 & \bf 0.639 & \bf 0.856 & \bf 0.639 & \bf 0.794 & \bf 0.908 & \bf 0.638 & \bf 0.922 & \bf 0.947\\
$\prod_{k=1}^\infty\left(1-p^{-2k+1}\right)$ & 0.793 & 0.639 & 0.855 & 0.639 &0.793 & 0.908 & 0.639 & 0.923 & 0.947\\
\bottomrule
\end{tabular}
\caption{Cyclicity frequencies for $10^5$ observations of $K_0(\cs(\g_{100,r}))_p$,  $p \mid r-1$. Marked in bold are empirical estimates supporting Conjecture~\ref{conjecture:regular}.} \label{table:newdist}
\end{center}
\end{table}

\begin{table}[!htbp]
\begin{center}
\begin{tabular}{c|c|c|ccccc}
\toprule
\multirow{2}{*}{$r=2^j+1$} &
\multirow{2}{*}{$\g_{100,r}$ connected} &
\multirow{2}{*}{$K_1\ne0$} &
\multicolumn{5}{|c}{$(\tor(K_0))_p$ cyclic} \\ \cline{4-8}
& & & all $p$ & $p=2$ & $p=3$ & $p=5$ & $p=7$\\
\hline
$3$ & 99519 & 57 & \bf 38494 & \bf 40599 & 95942 & 99198 & 99686\\
$5$ & 100000 & 0 & \bf 39460 & \bf 41781 & 95779 & 99152 & 99713\\
$9$ & 100000 & 0 & \bf 39600 & \bf 41978 & 95679 & 99221 & 99655\\
$17$ & 100000 & 0 & \bf 39655 & \bf 41903 & 95910 & 99212 & 99676\\
$33$ & 100000 & 0 & \bf 39879 & \bf 42133 & 95934 & 99121 & 99724\\
\hline
$n\to\infty$ & $10^5-O\left(\frac{1}{n}\right)$ & $o(1)$ & $10^5\gamma_{2^j+1}$ & $10^5\pi_{2,2^j+1}$ & 95851 & 99167 & 99702\\
\bottomrule
\end{tabular}
\caption{Selected totals for $10^5$ observations of $\cs(\g_{100,2^j+1})$, $1\le j\le 5$. Marked in bold are empirical estimates supporting Conjecture~\ref{conjecture:regular}.} \label{table:2k1}
\end{center}
\end{table}

\begin{table}[!htbp]
\begin{center}
\begin{tabular}{c|c|c|c|c|c|c|c}
\toprule
$r$ & Connected & $K_1\ne0$ & $K_0$ cyclic & $\mathrm{Det}<0$ & Full shift & $\ppm$ & $\oo_m$ \\
\hline
$3$ & 99495 & 65 & 38335 &  \textcolor{teal}{55987} & \textcolor{teal}{21448} & \textcolor{purple}{0} & \textcolor{purple}{0} \\
$4$ & 99989 & 0 & 26361 &  \textcolor{teal}{41056} &  \textcolor{teal}{10832} & \textcolor{purple}{0} & \textcolor{purple}{0} \\
$5$ & 100000 & 0 & 39611 &  \textcolor{teal}{58341} &  \textcolor{teal}{23108} & \textcolor{purple}{0} & \textcolor{purple}{0} \\
$6$ & 100000 & 0 & 31373 &  \textcolor{teal}{41669} &  \textcolor{teal}{13153} & \textcolor{purple}{0} & \textcolor{purple}{0} \\
$7$ & 100000 & 0 & 26325 &  \textcolor{teal}{52379} &  \textcolor{teal}{13864} & \textcolor{purple}{0} & \textcolor{purple}{0} \\
$8$ & 100000 & 0 & 34067 &  \textcolor{teal}{57153} &  \textcolor{teal}{19529} & \textcolor{purple}{0} & \textcolor{purple}{0} \\
$9$ & 100000 & 0 & 39401 &  \textcolor{teal}{47150} &  \textcolor{teal}{18616} & \textcolor{purple}{0} & \textcolor{purple}{0} \\
\hline
$n\to\infty$ & $10^5-O\left(\frac{1}{n}\right)$ & $o(1)$ & $10^5 \cdot \gamma_r$ & $\textcolor{teal}{10^5\varepsilon_{n,r}}$ & $\textcolor{teal}{10^5\tau_{n,r}}$ & \textcolor{purple}{0} & \textcolor{purple}{0} \\
\bottomrule
\end{tabular}
\caption{Selected totals for $10^5$ observations of $\cs(\g_{100,r})$, various $r$. Numbers in \textcolor{purple}{purple} provide empirical data for the open Question~\ref{q2}. Numbers marked in \textcolor{teal}{teal} provide empirical data for the open Question~\ref{q1}. In particular, using the definitions \eqref{eqn:gnr-shiftprob}, the values $\varepsilon_{n,r}$ and $\tau_{n,r}$ in the data table obey the relationship $\tau_{n,r}\approx\gamma_r\cdot\varepsilon_{n,r}$, supporting statistical independence of determinant negativity and cyclicity (cf.\ \eqref{eqn:gnrcuntz}).} \label{table:gnrextra}
\end{center}
\end{table}

\end{document}